\documentclass{article}

\usepackage{amssymb,amsmath, amsthm}

  \usepackage{graphics} 
  \usepackage{epsfig} 

\newcommand{\ti}{\;\,\makebox[0pt]{$\mid$}\makebox[0pt]{$\cap$}\,\;}

\newtheorem{theorem}{Theorem}[section]
\newtheorem{corollary}{Corollary}

\newtheorem{lemma}[theorem]{Lemma}
\newtheorem{proposition}{Proposition}

\newtheorem*{problem}{Problem}
\theoremstyle{definition}
\newtheorem{definition}[theorem]{Definition}
\newtheorem{remark}{Remark}

\title{On the indices of periodic points in $C^1$-generic wild homoclinic classes in dimension three}

\author{Katsutoshi Shinohara}

\date{}

\begin{document}

\maketitle

\begin{abstract}
We study the dynamics of homoclinic classes on three dimensional manifolds 
under the robust absence of dominated splittings.
We prove that if such a homoclinic class contains a volume-expanding periodic point, 
then, $C^1$-generically, it contains a hyperbolic periodic point 
whose index (dimension of the unstable manifold) is equal to two.   
\end{abstract}

\setcounter{section}{-1}
\section{Notations}

In this section, we collect some basic definitions which we frequently use 
throughout this article.

\subsection{Some basic terminologies}
We consider a closed (compact and boundaryless) smooth manifold $M$
with a Riemaniann metric. 
In this article, we mainly treat the case where $\mathrm{dim}M=3$.
The space of $C^1$-diffeomorphisms of $M$ is denoted by $\mathrm{Diff}^1(M)$.
We fix a distance function on $\mathrm{Diff}^1(M)$ derives from the Riemaniann metric and
furnish $\mathrm{Diff}^1(M)$ with the $C^1$-topology.
For $f \in \mathrm{Diff}^1(M)$, we denote 
the set of periodic points of $f$ by $\mathrm{Per}(f)$
and the set of hyperbolic periodic points of $f$ by $\mathrm{Per}_h(f)$.
For $P \in \mathrm{Per}(f)$, by $\mathrm{per}(P)$ we denote the period of $P$, 
i.e., the least positive integer $k$ that satisfies $f^k(P) =P$.
For $x \in M$ we denote the orbit of $x$ by $\mathcal{O}(x, f)$
or simply $\mathcal{O}(x)$.
We put $J(P) := \det(df^{\mathrm{per}(P)}(P))$ and 
call this value {\it Jacobian} of $P$.
We say that a periodic point is {\it volume-expanding} (resp. 
{\it volume-contracting, conservative}) if $|J(P)|>1$ 
(resp. $|J(P)|<1$, $|J(P)|=1$). 

In the following, we assume $P$ is a hyperbolic periodic point.
The index of a hyperbolic periodic point $P$ (denoted by $\mathrm{ind}(P)$)
is defined to be the dimension of the unstable manifold of $P$.
By $W^s(P, f)$ (resp. $W^u(P, f)$) we denote 
the stable (resp. unstable) manifold of $P$. 
We also use the simplified notation $W^s(P)$ (resp. $W^u(P)$). 
By $H(P, f)$, or simply $H(P)$, we denote the homoclinic class of $P$, i.e., the closure of the set of 
points of transversal intersections of $W^u(P)$ and $W^s(P)$.
Two hyperbolic periodic points $P$ and $Q$ are said to be {\it homoclinically related}
if $W^u(P)$ and $W^s(Q)$, $W^u(Q)$ and $W^s(P)$  both have non-empty
transversal intersections.
If $g$ is a $C^1$-diffeomorphism sufficiently close to $f$,
then one can define the continuation of $P$.
We denote the continuation of $P$ for $g$ by $P(g)$.
We say that $P$ has a {\it homoclinic tangency} 
if there exitsts a point of non-transversal intersection
between stable manifold and unstable manifold of $P$,
i.e., there exists a point $x \in W^s(P) \cap W^u(P)$ at which $T_xW^s(P)$ and
$T_xW^s(P)$ do not span $T_xM$.

\subsection{Linear cocycles and dominated splittings}\label{ddomi}
Let $(\Sigma, f, E, A)$ be a {\it linear cocycle}, 
where $\Sigma$ is a topological space, $f$ is a homeomorphism 
of $\Sigma$, $E$ is a Riemaniann vector bundle over $\Sigma$ and
$A$ is a bundle map that is compatible with $f$, i.e., 
$A$ is a map $A \colon E \to E$, where for each $x \in \Sigma$,
$A(x, \, \cdot \, ) $ is the linear isomorphisms from $E(x)$ to $E(f(x))$.  
We also use the notation $A(x)$ in the sense of $A(x, \, \cdot \,)$
and denote the linear cocycle only with $E$ or $A$
when the meaning is clear from the context.
In our application, we mainly treat linear cocycles where $\Sigma$
is some invariant set of $M$, $f$ is the restriction of a diffeomorphism to $\Sigma$, 
$E$ is the restriction of the tangent bundle to $\Sigma$, 
and $A$ is the differential of $f$ restricted to $E$. 
We can naturally define the $n$-times iteration of $A$, denoted by $A^n$,
and inverse of $A$, denoted by $A^{-1}$. 
We say that a linear cocycle $(\Sigma, f, E, A)$ is {\it periodic} if each point 
$x \in \Sigma$ is periodic for $f$.
A periodic linear cocycle is said to be {\it diagonalizable} if 
at each point, the return map (the composition of corresponding 
linear maps along the orbit) is diagonalizable.

On each fiber, there is a norm that derives from the Riemaniann metric.
We denote it by $|| \,\, \cdot \,\, ||$.
Then, a linear cocycle is said to be {\it bounded by $K >0$} if the following inequality holds:
\[
\max \left\{ \sup_{x \in \Sigma} ||A(x)||, \,\, \sup_{x \in \Sigma} ||A^{-1}(x)|| \right\} <K.
\] 
When $A$ is the restriction of the differential of a diffeomorphism on
a compact manifold, 
$A$ is bounded by some constant.
For a linear cocycle, we can canonically define invariant subcocycle,
direct sum between some cocycles, and quotient of cocycles 
(for the details, see section 1 of \cite{BDP}).

Let $(\Sigma, f, E, A)$  be a linear cocycle and suppose 
$E$ is a direct sum of 
two non-trivial 
invariant subbundles $(\Sigma, f, F, A|_{F})$ and $(\Sigma, f, G, A|_{G})$
(denoted by $E= F\oplus G$), 
where $A|_{F}$ is the restriction of $A$ to $F$. 
For a positive integer $n$, 
we say that $F\oplus G$ is an {\it $n$-dominated splitting} if following holds:
\[
\| A^n(x)|_{F} \| \|A^{-n}(f^{n}(x))|_{G} \| < 1/2, \,\,\, \mbox{for all} \,\,\,  x\in \Sigma.
\]
We say that a linear cocycle $(\Sigma, f, E, A)$ admits a dominated splitting 
if there exists two invariant subbundles $F$, $G$ of $E$ and an integer $n$ 
such that $E= F \oplus G$ is an $n$-dominated splitting.
We say that an $f$-invariant set $\Lambda \subset M$ admits a dominated splitting if 
the linear cocycle $(\Lambda, f, TM|_{\Lambda}, df)$ admits a dominated splitting.

\subsection{Chain reccurence class and robust cycles}
Let $\varepsilon$ be a positive real number, $x, y \in M$ and
$d( \, \cdot \, , \, \cdot \, )$ be a metric on $M$.
An {\it $\varepsilon$-chain from $x$ to $y$} 
for $f \in \mathrm{Diff}^1(M)$ 
is a sequence $(x_i)_{i=1}^n$ ($n \geq 2$) in $M$ satisfying
$d( f(x_i) , x_{i+1} ) < \varepsilon$ for $1\leq  i < n$, 
$x_1 =x$ and $x_n =y$. 
Two points $x, y$ are said to be {\it chain equivalent}
if for any $\varepsilon >0$ there exist an $\varepsilon$-chain 
from $x$ to $y$ and an  $\varepsilon$-chain from $y$ to $x$.
A point $x \in M$ is said to be a {\it chain recurrent point}
if $x$ is chain equivalent to itself.
For a chain recurrent point $x$, its {\it chain recurrence class} of $x$
is defined to be the set of the points that is chain equivalent to 
$x$. 

Let $\Gamma$ and $\Sigma$ be two transitive hyperbolic
invariant sets for $f$.  We say that $f$ $\Gamma$ and $\Sigma$ have a {\it heterodimensional cycle}
if the following holds:
\begin{enumerate}
\item The indices (the dimension of the unstable manifolds) of the sets $\Gamma$ and $\Sigma$ are different.
\item The stable manifold of $\Gamma$ meets the unstable manifold of $\Sigma$ and 
the same holds for stable manifold of $\Sigma$ and the unstable manifold of $\Gamma$.
\end{enumerate}
We say that the heterodimensional cycle associated to $\Gamma$ and $\Sigma$ is $C^1$-robust
if there exists a $C^1$-neighborhood $\mathcal{U}$ of $f$ such that for each $g \in \mathcal{U}$
there exists a heterodimensional cycle associated to
the continuations $\Gamma(g)$ of $\Gamma$ and $\Sigma(g)$ of $\Sigma$.

\section{Introduction}

Let $M$ be a compact smooth manifold without boundary.
For $f \in \mathrm{Diff}^1(M)$ and $P \in \mathrm{Per}_h(f)$, 
its {\it homoclinic class}, denoted by $H(P, f)$ (or $H(P)$),
is defined to be the closure of the set of the points 
of the transversal intersection between 
the stable manifold and the unstable manifold of $P$. 
The theory of Smale's generalized horseshoe tells us that $H(P)$ coincides with 
the closure of the set of 
hyperbolic periodic points that are homoclinically related to $P$.
In the study of uniformly hyperbolic systems, homoclinic classes play an important role
and it is expected that they also play an important role in the research of non-hyperbolic dynamics (see chapter 10 of \cite{BDV}). 

In the Axiom A diffeomorphisms, every homoclinic class exhibits 
uniformly hyperbolic structure. That enables us to investigate the fine internal structure 
of dynamics of homoclinic classes. However, in the non-uniformly hyperbolic
systems, homoclinic classes do not necessarily exhibit uniform hyperbolicities.
This fact makes the study of non-hyperbolic dynamics difficult.

Even in the non-uniformly hyperbolic systems, a homoclinic class
may exhibit weak form of hyperbolicity, such as partial hyperbolicity
(see \cite{BDV} for definition) or dominated splitting 
(see section \ref{ddomi} for definition). 
On the other hand, there do exist homoclinic classes that do not exhibit 
any kind of hyperbolicities in a robust fashion,
and there are several indications that 
such kind of absence of hyperbolicities implies the complexities of the dynamics.  
For example, 
Bonatti, D\'{i}az, and Pujals \cite{BDP} proved that the robust absence of 
the dominated splitting on a homoclinic class implies the $C^1$-Newhouse phenomenon,
i.e., locally generic coexistence of infinitely many sinks or sources.
Furthermore, in \cite{BD}, Bonatti and D\'{i}az showed, under the robust absence of dominated splitting 
and some conditions on the Jacobians, 
a homoclinic class exhibits very complicated 
dynamics named {\it universal dynamics}.

Thus, the following questions are interesting:
{\it What are the effects that the absence of 
dominated splitting on a homoclinic class gives rise to?} 
Or, {\it how the existence of the dominated splitting 
on a homoclinic class is disturbed? }
There are some results in this direction. For example,
Gourmelon \cite{GouH} proved that under the absence of dominated splittings
on a homoclinic class, one can create a homoclinic tangency inside the homoclinic class. 
Note that Wen \cite{W} also proved similar result starting from preperiodic points.
The result of Gan says 
that the existence of dominated splitting of index $i$ on preperiodic points
is equivalent to the existence of $i$-eigenvalue gap (see \cite{Gan} for precise definition).
In this article, inspired by Abdenur, et al. (\cite{ABCDW}), 
we investigate the {\it index set} of homoclinic classes that do not 
admit dominated splittings from the $C^1$-generic viewpoint.

To review the result of \cite{ABCDW}, we prepare one notation.
Given $f \in \mathrm{Diff}^1(M)$ and $P \in \mathrm{Per}_h(f)$,
the {\it index set} of $H(P, f)$ (denoted by $\mathrm{ind}(H(P, f))$ 
is defined to be the set of integers that
appear as a index of some periodic points in $H(P, f)$,
i.e., we put
\[
\mathrm{ind}(H(P, f)) := \{ k \in \mathbb{N} \mid \exists Q \in 
\mathrm{Per}_h(f) \cap H(P, f), \,\, \mathrm{ind}(Q) = k \}.
\]
In \cite{ABCDW}, it was proved that for $C^1$-generic diffeomorphism,
every homoclinic class has index set which is an interval in $\mathbb{N}$.
The problem we pursue in this article is the following:
{\it Does the robust absence of dominated splittings give some 
restrictions on its index set?}

Let us state our main result.
For $f \in \mathrm{Diff}^1(M)$ and $P \in \mathrm{Per}_h(f)$, 
$H(P)$ is is said to be {\it wild} if it is robustly non-dominated,
more precisely, there exists a neighborhood 
$\mathcal{U} \subset \mathrm{Diff}^1(M)$ of $f$ such that 
for every $g\in \mathcal{U}$, the continuation $P(g)$ is defined and
$H(P, g)$ does not admit any kind of dominated splittings.

Here is a partial answer to this question.
\begin{theorem}\label{mainth}
For $C^1$-generic diffeomorphisms of a three-dimensional closed smooth manifold,
if there exists a wild homoclinic class $H(P)$ that contains an index-one 
volume-expanding 
hyperbolic periodic point, then $2 \in \mathrm{ind}(H(P))$.
\end{theorem}

Let us see an immediate corollary of this theorem.
\begin{corollary}
For $C^1$-generic diffeomorphisms of a three-dimensional closed smooth manifold, 
if there exists a wild homoclinic class that contains two hyperbolic periodic points 
such tha one of them is volume-expanding and the other is volume-contracting.  
Then $\mathrm{ind}(H(P)) = \{1, 2 \}$.
\end{corollary}
Thus, under some assumptions on Jacobian, 
there does exist a restriction on the index of periodic points inside wild homoclinic classes.

This theorem can be interpreted as a qualification of homoclinic classes to be ``basic pieces.''
To explain this, let us review the intuitive idea of \cite{BDP}:
The wildness of a homoclinic class scatters its hyperbolicity to any direction.
Thus by using their technique, it is not difficult to prove that, under the wildness,  
one can create an index bifurcation by an arbitrarily small perturbation. 
However, this argument does not tell us whether 
the bifurcation happens inside the homoclinic class or not.
If homoclinic classes are to deserve as basic pieces, then it is desireble that 
a phenomenon which local (linear algebraic) information implies to happen
can be observed inside the original homoclinic classes.
Theorem \ref{mainth} says that the bifurcation problem we discussed above 
is ``well localized'' in homoclinic classes.

Let us reintroduce our theorem from a different viewpoint.
Aiming at the global understanding of $C^1$-dynamical systems,
Palis suggested the famous Palis conjecture (see \cite{P}), that is,
every ($C^1$-) diffeomorphism away from Axiom A and no-cycle diffeomorphisms
can be approximated by a diffeomorphism with heterodimensional cycle or homoclinic tangency.
Recently, Bonatti and D\'{i}az asked a stronger version of this conjecture:
\begin{problem}[Question 1.2 in \cite{BD2}]
Let $M$ be a smooth closed manifold. Does there exist a $C^1$-open and dense subset 
$\mathcal{O} \subset \mathrm{Diff}^1(M)$ such that every $f \in \mathcal{O}$
either verifies the Axiom A and the no-cycle condition or has a $C^1$-robust 
heterodimensional cycle.
\end{problem}

Our study gives a partial answer to this question.
In fact, we can prove the following:
\begin{theorem}\label{nPalis}
Let $f$ be a $C^1$-diffeomorphism of a three-dimensional closed smooth manifold.
If $f$ has a wild homoclinic class that contains an index-1 
volume-expanding hyperbolic periodic point,
then $f$ can be approximated by a diffeomorphism with a robust heterodimensional cycle.
\end{theorem}
Note that Wen \cite{W} and Gourmelon \cite{GouH} already gave 
positive answers to the Palis conjecture under similar hypothesis.
The novelity of our results is that we can create a {\it connection between two saddles}.
Roughly speaking, outside Axiom A diffeomorphisms with no-cycles, 
by linear algebraic arguments and Franks' lemma,
it is not difficult to create an index bifurcation with an arbitrarily small perturbation.
On the other hand, in general it is difficult to create a cycle between two saddles,
since we need the information about the reccurence between two saddles. 
Our proof suggests one senario to create a connection between two saddles.

In this article, we confined our attention to three-dimensional cases.
Let us see what happens the other dimensions.
In dimension two, it is easy to determine $\mathrm{ind}(H(P))$. It is $\{ 0\}$ (sink), 
$\{ 2\}$ (source), or $\{ 1 \}$. Note that  
there is no example of wild homoclinic classes in dimension two (in $C^1$-topology). 
If one can construct such a homoclinic class,
it serves as a counter-example of the conjecture of Smale about the density of the
Axiom A and no-cycle condition diffeomorphisms in $C^1$-topology (see \cite{S}).
The study of higher dimensional cases are of natural interest.
So far, the complete solution of the higher dimensional cases are not obtained.
In \cite{Sh}, the author gave an example of generically wild homoclinic classes 
on four dimensional manifold whose index set exhibits an {\it index deficiency}
in a robust fashion (for the precise definition, see \cite{Sh}). 

The example of the homoclinic class that satisfies the assumption of the theorem
can be found in \cite{BD}. As far as I know, this is the only mechanism that assures the wildness of some homoclinic classes.
Thus it is interesting to study whether there is another mechanism that leads the creation of a wild homoclinic class.
For example, the following question is interesting to study:
\begin{problem}
Can one find $f \in \mathrm{Diff}^1(M)$ ($\mathrm{dim}(M) =3$) 
such that for some hyperbolic periodic point $P$, 
$H(P)$ is wild and robustly $\mathrm{ind}(H(P)) = \{ 1\}$?
\end{problem}

Finally, let us see the organization of this article. 
In section \ref{outline}, we introduce our strategy for the proofs of Theorem \ref{mainth} and \ref{nPalis}. 
We also furnish some part of their proof. The rest of the proofs are given
in section \ref{cretan} and section \ref{biftan}.
More explanation on the contents of section \ref{cretan} and 
section \ref{biftan} can be found at the end of the section \ref{outline}.

\section{Outline of the proof}\label{outline}

In this section, we see the strategy 
for the proof of Theorem \ref{mainth} and \ref{nPalis}.
First, we explain how we will prove Theorem \ref{nPalis}.
The proof is divided into two propositions. 

We say a hyperbolic periodic point $P$ has a {\it homothetic tangency}
if $P$ has a homoclinic tangency and the restrictions of $df^{\mathrm{per}(P)}$
to the stable spaces and unstable spaces are both homotheties
(a linear endomorphism of a linear space is said to be a 
{\it homothety} if it is equal to $r \mathrm{Id}$, where $r$ is a positive
real number and $\mathrm{Id}$ is the identity map).

Roughly speaking, the first proposition states that, under the 
robust absence of dominated splittings, one can create a 
homothetic tangency
inside the homoclinic class by an arbitrarily small perturbation. 
\begin{proposition}\label{pdegtan} 
Let $f \in \mathrm{Diff}^1(M)$ with $\dim M =3$ and
$P$ be a volume-expanding index-$1$ hyperbolic periodic point of $f$.
If $H(P)$ is wild then  one can find 
a $C^1$-diffeomorphism $g$ arbitrarily $C^1$-close to $f$ such that the following properties hold.
\begin{enumerate}
\item There exists a volume-expanding hyperbolic periodic point $Q$ of index $1$. 
\item Let $P(g)$ be the continuation of $P$. Then two periodic points $P(g)$ and $Q$ are homoclinically related.
\item $Q$ has a homothetic tangency.
\end{enumerate}
\end{proposition}

The second proposition says that from a homothetic tangency 
one can create a heterodimensional cycle by an arbitrarily $C^1$-small perturbation.

\begin{proposition}\label{pbiftan}
Let $f \in \mathrm{Diff}^1(M)$ with $\dim M =3$ and
$Q$ be a volume-expanding hyperbolic periodic point of $f$ with $\mathrm{ind}(Q)=1$.
If  $Q$ has a homothetic tangency,
then one can find 
a $C^1$-diffeomorphism $g$ arbitrarily $C^1$-close to $f$ such that the following properties hold.
\begin{enumerate}
\item There exists a hyperbolic periodic point $R$ of $g$ with $\mathrm{ind}(R)=2$. 
\item $g$ has a heterodimensional cycle associated to
two periodic points $Q(g)$ and $R$.
\end{enumerate}
\end{proposition}

We need the following results in \cite{BDKS}.

\begin{lemma}[Theorem 1 in \cite{BDKS}]\label{BDKS}
Let $f$ be a $C^1$-diffeomorphism  with a heterodimansional cycle
associated to saddles $Q$ and $R$ with $\mathrm{ind}(Q) - \mathrm{ind}(R) =\pm1$.
Suppose that at least one of the homoclinic classes of these saddles is non-trivial.
Then there are diffeomorpshims $g$ arbitrarily $C^1$-close to $f$ with robust heterodimensional cycles 
associated to two transitive hyperbolic sets containing the continuations $Q(g)$ and $R(g)$.
\end{lemma}
We can summarize the results of Proposition \ref{pdegtan},
Proposition \ref{pbiftan}, and Lemma \ref{BDKS} as follows:
\begin{proposition}\label{pw-hdc}
Let $f \in \mathrm{Diff}^1(M)$ with $\dim M =3$ and
$P$ be a volume-expanding index-$1$ hyperbolic periodic point of $f$.
If $H(P)$ is wild then one can find 
a $C^1$-diffeomorphism $g$ arbitrarily $C^1$-close to $f$ 
and a hyperbolic periodic point $R$ of $g$ with $\mathrm{ind}(R) =2$
such that $g$ has a robust heterodimensional cycle
associated to two transitive hyperbolic sets containing $P(g)$ and $R$.
\end{proposition}
It is clear that Proposition \ref{pw-hdc} implies Theorem \ref{nPalis}.

Now, let us see how we prove Theorem \ref{mainth}
using Proposition \ref{pw-hdc}.
For the proof, we need some generic properties about $C^1$-diffeomorphisms
listed below.

\begin{enumerate}
\def\labelenumi{($\mathcal{R}_{\theenumi}$)}
\item We denote the set of Kupka-Smale diffeomorphisms by $\mathcal{R}_1$
and the set of Kupka-Smale 
diffeomorphisms such that none of the periodic points
are conservative by $\mathcal{R}'_1$. 
These are residual sets in $\mathrm{Diff}^1(M)$. 
The genericity of $\mathcal{R}_1$ is well known.
One can prove the genericity of $\mathcal{R}'_1$ by modifying the 
usual proof of Kupka-Smale theorem (see \cite{R}, for example) 
with paying attention to the fact that 
the volume-conservativeness of a hyperbolic periodic point is a fragile
property (i.e., can be destroyed by an arbitrarily $C^1$-small perturbation).

\item By $\mathcal{R}_2$ we denote the set of diffeomorphisms
such that following holds: Any chain reccurence class $C$ containing 
a hyperbolic periodic point $P$ satisfies $C =H(P)$. See \cite{BC}.

\item By $\mathcal{R}_3$ we denote the set of the diffeomorphisms $f$ that 
satisfying the follwoing: Let $P$ and $Q$ be 
hyperbolic periodic points with $\mathrm{ind}(P) = \mathrm{ind}(Q)$.
If $H(P) \cap H(Q) \neq \emptyset$ then $P$ and $Q$ are homoclinically related.
We will give the proof of the genericity of  $\mathcal{R}_3$ later.
\end{enumerate}

The following lemma is easy to prove, so we omit the proof.
\begin{lemma}\label{l-easy}
Let $P$ and $Q$ be hyperbolic periodic points and 
assume there exists a heterodimensional cycle associated to two 
transitive hyperbolic invariant sets $\Gamma$ and $\Sigma$
such that $\Gamma$ contains $P$ and $\Sigma$ contains $Q$. 
Then, $P$ and $Q$ belong to the same chain recurrence class.
\end{lemma}

Let us give the proof of Theorem \ref{mainth} using Proposition \ref{pw-hdc}.
\begin{proof}[Proof of Theorem \ref{mainth}]
For $f \in \mathrm{Diff}^1(M)$,
let $\mathrm{Per}^N_{h}(f)$ be the set of hyperbolic 
periodic points whose periods are less than $N$.
For every $f \in \mathcal{R}'_{1}$, we take an open neighborhood 
$\mathcal{U}_N(f) \subset \mathrm{Diff}^1(M)$  of $f$ 
such that every $g \in \mathcal{U}_N(f)$ satisfies the following conditions:
\begin{enumerate}
\item For each $P_i \in \mathrm{Per}^N_{h}(f)$, one can define the continuation $P_i(g)$ (in particular we require that every $P_i$ does not exhibit
index bifurcation in $\mathcal{U}_N(f)$).
\item There is no creation of periodic points with period less than $N$, more precisely, 
the set $\{ P_i(g)\}$ exhausts the periodic points 
whose periods are less than $N$.
\item For each $P_i \in \mathrm{Per}^N_{h}(f)$, the signature of $\log |J(P_i(g))|$ is independent of the choice of $g$.
\end{enumerate}
Note that since $f$ is Kupka-Smale, $\mathrm{Per}^N_{h}(f)$ is a finite set.
From now on, for each $\mathcal{U}_N(f)$, 
we create an open and 
dense set $\mathcal{O}_{N} \subset \mathcal{U}_N(f)$ as follows. 
First, for each $g \in \mathcal{U}_N(f)$, we define an open set 
$\mathcal{V}_{i,N}[g] \subset \mathcal{U}_N(f)$ as follows: 
\begin{enumerate}
\item If $P_i$ satisfies one of the following conditions; 
\begin{itemize} 
\item $\mathrm{ind}(P_i) \neq 1$,
\item volume-contracting (by assumption $P_i$ cannot be conservative), or
\item arbitrarily $C^1$-close to $g$ there exists 
a $C^1$-diffeomorphism $g'$ such that $H(P_i,g')$ admits a dominated splitting.
\end{itemize}
Then $\mathcal{V}_{i,N}[g] := \emptyset$. 
\item Otherwise, $\mathcal{V}_{i,N}[g]$ is the open set 
in $\mathcal{U}_N(f)$ obtained by Proposition \ref{pw-hdc}.
\end{enumerate}
Then, put $\mathcal{V}_{i,N}(f) := \cup _{g \in \mathcal{U}(f)} \mathcal{V}_{i,N}[g]$.
Since each $\mathcal{V}_{i,N}[g]$ are open, $\mathcal{V}_{i,N}(f)$ is open.
Let us put $\mathcal{W}_{i,N}(f) := \mathcal{U}_N(f) \setminus \overline{\mathcal{V}_{i,N}(f)}$
and $\mathcal{O}_{i,N}(f) := \mathcal{V}_{i,N}(f) \cup \mathcal{W}_{i,N}(f) $.
By construction, $\mathcal{O}_{i,N}(f)$ is open and dense in $\mathcal{U}_N(f)$ and
for every $g \in \mathcal{O}_{i,N}(f)$, if $P_i(g)$ is 
a volume-expanding, hyperbolic and index-1 periodic point and $H(P_i,g)$ is wild,
then $g$ has a robust cycle associated with transitive hyperbolic sets containing
$P_i(g)$ and some hyperbolic periodic point of index $2$.

Now, put $\mathcal{O}_N(f) := \cap _i \mathcal{O}_{i,N}(f)$. Since each 
$\mathcal{O}_{i,N}(f)$ is open and dense, 
so is $\mathcal{O}_N(f)$ in $\mathcal{U}_N(f)$ and the set 
$\mathcal{O}_N := \cup _{f \in \mathcal{R}_{1}} \mathcal{O}_N(f)$ is the 
required open and dense subset of $\mathrm{Diff}^1(M)$.
Then, put $\mathcal{R}_{\ast} := \cap_{N} \mathcal{O}_N$. 
By Baire's category theorem this is a residual subset of $\mathrm{Diff}^1(M)$ and 
satisfies the following property: If $f \in \mathcal{R}_{\ast}$ and there exists a hyperbolic, index-1 and volume-expanding 
periodic point $P$ whose homoclinic class is wild, 
then one can find a robust cycle associated with $P$ and
some periodic point with index $2$.

Finally, put 
$\mathcal{R}_{\ast \ast} := \mathcal{R}_{\ast} \cap \mathcal{R}_2 \cap \mathcal{R}_3$.
We show that every $f \in \mathcal{R}_{\ast \ast}$ satisfies the conclusion of Theorem \ref{mainth}.
Suppose that there is a hyperbolic periodic point $P$ of $f$ such that
$H(P)$ is a wild homoclinic class containing 
a volume-expanding hyperbolic periodic point $Q$ with $\mathrm{ind}(Q)=1$.
If $\mathrm{ind}(P)=2$, then we have the conclusion.
So let us assume $\mathrm{ind}(P)=1$.
Since $f \in \mathcal{R}_3$, $P$ and $Q$ are homoclinically related.
For $H(P)$ is wild, so is $H(Q)$.
Then, by the definition of  $\mathcal{R}_{\ast}$,
$f$ has a heterodimensional cycle associated to two transitive hyperbolic set
such that one of them contains $Q$ and the other contains a hyperbolic 
periodic point $R$ of index $2$. 
By Lemma \ref{l-easy}, $Q$ and $R$ belong to the same chain recurrence class.
Since $ f \in \mathcal{R}_2$, we have $R \in H(Q)=H(P)$.
\end{proof}

\begin{remark}
Indeed, the proof above says that, 
for $C^1$-generic diffeomorphisms of a three-dimensional closed smooth manifold,
every wild homoclinic class $H(P)$ that contains a index-one volume-expanding 
hyperbolic periodic point contains a robust heterodimensional cycle. 
\end{remark}

Now, we give the proof of genericity of $\mathcal{R}_3$.
For the proof, we need another generic property and one $C^1$-perturbation lemma.
\begin{lemma}[Lemma 2.1 in \cite{ABCDW}]\label{lloc-cnt}
There is a residual subset $\mathcal{R}_4 \subset \mathrm{Diff}^1(M)$
such that, for every diffeomorphism $f \in \mathcal{R}_4$ and every 
pair of saddles $P(f)$ and $Q(f)$ of $f$, there is a neighborhood $\mathcal{U}_f$
of $f$ in $\mathcal{R}_4$ such that either $H(P,g) = H(Q,g)$ for all $g \in \mathcal{U}_f$,
or $H(P,g) \cap H(Q,g) = \emptyset$ for all $g \in \mathcal{U}_f$.
\end{lemma}

\begin{lemma}[See Lemma 2.8 in \cite{ABCDW}]\label{l-cnn}
Let $P$ be a hyperbolic periodic point of a $C^1$-diffeomorphism $f$.
Consider a homoclinic class $H(P,f)$ and any saddle $Q \in H(P, f)$.
Then there is $g$ arbitrarily $C^1$-close to $f$ such that 
$W^s(P, g)$ and $W^u(Q, g)$ (resp. $W^u(P, g)$ and $W^s(Q, g)$) have 
a non-empty intersection.
\end{lemma}

\begin{proof}[Proof of the genericity of $\mathcal{R}_3$]
For every $f \in \mathcal{R}_{0} \cap \mathcal{R}_4$, we take an open neighborhood 
$\mathcal{U}'_N(f) \subset \mathcal{R}_{0} \cap \mathcal{R}_4$  of $f$ 
such that for every $g \in \mathcal{U}'_N(f)$ we have the following conditions:
For each $P_i \in \mathrm{Per}^N_{h}(f)$, one can define the continuation $P_i(g)$
and these continuations exhaust periodic points whose periods are less than $N$ of $g$ (note that $\mathrm{Per}^N_{h}(f)$ is a finite set since $f \in \mathcal{R}_{0}$).
Now, for each $\mathcal{U}'_N(f)$, we construct an open 
and dense set $O_{N}(f)  \subset \mathcal{U}'_N(f)$ as follows. 
First, for each $g \in \mathcal{U}'_N(f)$ and $i \neq j$, let us define an open 
set $\mathcal{V}_{(i,j),N}[g] \subset \mathcal{U}'_N(f)$ as follows: 
if $P_i(g)$ and $P_j(g)$ are not homoclinically related, 
then $\mathcal{V}_{(i,j),N}[g]$ is the empty set. 
Otherwise $\mathcal{V}_{(i,j),N}[g] \subset \mathcal{U}'_N(f)$ is a non-empty open set 
satisfying the following: If $h \in \mathcal{V}_{(i,j),N}[g]$
then $P_i(h)$ and $P_j(h)$ are homoclinically related.
Then, let $\mathcal{V}_{(i,j),N}(f) := \cup _{g \in \mathcal{U}'_N(f)} \mathcal{V}_{(i,j),N}[g]$.
We put $\mathcal{W}_{(i,j),N}(f) := \mathcal{U}'_N(f) \setminus \overline{V_{(i,j),N}(f)}$
and $\mathcal{O}_{(i,j),N}(f) := {\mathcal{V}_{(i,j),N}(f)} \cup {\mathcal{W}_{(i,j),N}(f)}$.
By construction, $\mathcal{O}_{(i,j),N}(f)$ is open and dense in $\mathcal{U}'_N(f)$.

Let us see that for $h \in \mathcal{O}_{(i,j),N}$ following holds:
If $P_i  \in H(P_j , h)$ and $\mathrm{ind}(P_i) = \mathrm{ind}(P_j)$,
then $P_i$ and $P_j$ are homoclinically related.
If $h \in \mathcal{V}_{(i,j),N}(f)$, this is clear.
We show that for $h  \in  \mathcal{W}_{(i,j),N}$, we have $P_i \not \in H(P_j , h)$.
Indeed, if $P_i \in H(P_j , h)$, by applying Lemma \ref{lloc-cnt},
we can take a neighborhood $\mathcal{C}(h) \subset \mathcal{U}'_N(f)$ of $h$
such that for all $h' \in \mathcal{C}(h)$,  $H(P_i, h') = H(P_j, h')$.
Then, Lemma \ref{l-cnn} tells us there exists $h_1$ arbitrarily $C^1$-close to $h$
such that $W^u(P_i, h_1)$ and $W^s(P_j, h_1)$ have non-empty intersection.
By giving arbitrarily small perturbation, we can find a diffeomorphism $h_2$
such that $W^u(P_i, h_2)$ and $W^s(P_j, h_2)$ have 
non-empty transversal intersection.
Since it is a $C^1$-open property,
we can find $h_3 \in \mathcal{C}(h)$ arbitrarily $C^1$-close to $h_2$ such that  
$W^u(P_i, h_2)$ and $W^s(P_j, h_2)$ have non-empty transversal intersection
(note that $h_1$, $h_2$ can fail to belong to $\mathcal{C}(h)$).
By a similar argument, we can find $h_4$ arbitrarily $C^1$-close to $h_3$
such that $P_i(h_4)$ and $P_j(h_4)$ are homoclinically related.
This is a contradiction, since $h \in \mathcal{W}_{(i,j),N}(f) 
= \mathcal{U}_N(f) \setminus \overline{\mathcal{V}_{(i,j),N}(f)}$
and $h_4$ can be found arbitrarily $C^1$-close to $h$.

We put $\mathcal{O}_{N}(f) := \cap_{i \neq j} \mathcal{O}_{(i,j),N}(f)$.
Since $f_0 \in \mathcal{R}_0 $, 
the number of the periodic points with period less than $N$ are finite.
So $\mathcal{O}_N(f)$ is an open and dense subset of $\mathcal{U}'_N(f)$.
Now we define $\mathcal{U}'_N := \cup_{f \in \mathcal{R}_0}\mathcal{O}_N(f)$
and this is an open and dense subset of $\mathrm{Diff}^1(M)$.
Finally, we put $\mathcal{R}_3 := \cap_{N} \mathcal{U}'_N$.
Then, by construction, every diffeomorphism in $\mathcal{R}_3$
satisfies the desired condition and by Baire's category theorem
this set is residual in $\mathrm{Diff}^1(M)$.
\end{proof}

In the rest of this section, we discuss the proofs of Proposition \ref{pdegtan} and \ref{pbiftan}. 
In Section \ref{cretan} we give the proof of Proposition \ref{pdegtan}.
It is done by three techniques. The first one is the linear algebraic arguments developed in \cite{BDP}.
We conbine this technique with the second one, 
the Franks' lemma that preserves the invariant manifolds \cite{GouF}.
The third one is the result given by Gourmelon \cite{GouH} about  
the creation of homoclinic tangency for homoclinic classes
that does not admit dominated splittings. 

In Section \ref{biftan}, we give the proof of Proposition \ref{pbiftan}.
The proof consists of two steps. 
The first one is the reduction of the probrem to affine dynamics.
The second one is the investigation of the reduced dynamics involving 
some calculations.

\section{Creation of a homothetic tangency}\label{cretan}

In this section we prove Proposition \ref{pdegtan}. 

\subsection{Strategy for the proof of Proposition \ref{pdegtan}}
An important step to Proposition \ref{pdegtan} is the following Proposition.
\begin{proposition}\label{p1degtan}
Let $f \in \mathrm{Diff}^1(M)$ with $\dim M =3$, and 
let $P$ be a volume-expanding hyperbolic periodic point of $f$ such that
$\mathrm{ind}(P)=1$ and $H(P)$ is wild.
Then one can find a $C^1$-diffeomorphism $g$ arbitrarily $C^1$-close to $f$  such that following holds:
\begin{enumerate}
\item There exists a volume-expanding hyperbolic periodic point $Q$ of index $1$.
\item The differential $dg^{\mathrm{per}(Q)}(Q) $ has only positive and real eigenvalues.
\item Two periodic points $P(g)$ and $Q$ are homoclinically related.
\item The differential $dg^{\mathrm{per}(Q)}(Q)$ restricted to the stable direction of $Q$ is a homothety.
\end{enumerate}
\end{proposition}

By this proposition together with the following result given by Gourmelon,
we can prove Proposition \ref{pdegtan}.
\begin{lemma}[Theorem 1.1 in section 6 of \cite{GouH}]\label{tgouh}
If the homoclinic class $H(P,f)$ of a saddle point $P$ for $f$ is not trivial and does not 
admit a dominated splitting of the same index as $P$. Then, there is an 
arbitrarily small perturbation $g$ of $f$, that perserves the dynamics on a neighborhood of $P$,
and such that there is a homoclinic tangency associated to $P$.
\end{lemma}

Let us give the proof of Proposition \ref{pdegtan} assuming above two results.

\begin{proof}[Proof of Proposition \ref{pdegtan}]
Under the hypothesis of Proposition \ref{pdegtan}, 
Proposition \ref{p1degtan} tells us that we get $f_1$ arbitrarily $C^1$-close to $f$ 
such that $f_1$ has a hyperbolic periodic points $Q(f_1)$ satisfying all the conclusions of Proposition \ref{p1degtan}.
By taking $f_1$ sufficiently close to $f$, we can assume $H(P, f_1)=H(Q, f_1)$ does not admit dominated splittings.
Then, by applying Lemma \ref{tgouh} to $Q(f_1)$ we get $f_2$ arbtirarily close to $f_1$ such that 
$Q(f_1)=Q(f_2)$ exhibits a homoclinic tangency. 
Since the perturbation preserves the local dynamics of $Q(f_1)$,
$df_2^{\mathrm{per}(Q(f_2))}(Q(f_2)) = df_1^{\mathrm{per}(Q(f_1))}(Q(f_1))$ and
$Q(f_2)$ is volume-expanding. Thus we have created a homothetic tangency.
Furthermore, since $P(f_1)$ and $Q(f_1)$ are homoclinically related, 
if we take $f_2$ sufficiently close to $f_1$, we know $P(f_2)$ and $Q(f_2)$
are homoclinically related, too.
Note that we can take $f_2$ arbitrarily close to $f$ because $f_1$ can be found arbitrarily close to $f$.
Now the proof is completed.
\end{proof}
Thus let us concentrate on the proof of Proposition \ref{p1degtan}.
We divide the proof into two lemmas.
\begin{lemma}\label{lcomp}
Let $f \in  \mathrm{Diff}^1(M)$ with $\dim M=3$ 
and $P$ be a volume-expanding hyperbolic periodic point of $f$ 
and $\mathrm{ind}(P)=1$. If $H(P)$ is wild, then $C^1$-arbitrarily close to $f$ one can find 
a $C^1$-diffeomorphism $g$ such that following holds:  
There exists a volume-expanding hyperbolic index-1 periodic point $Q$ such that
$P(g)$ and $Q$ are homoclinically related, and the restriction of $dg^{\mathrm{per}(Q)}(Q) $ 
to the stable direction has two complex eigenvalues.
\end{lemma}
\begin{lemma}\label{lhomo}
Let $f \in  \mathrm{Diff}^1(M)$ with $\dim M=3$ and  $P$ be an index-1 volume-expanding hyperbolic 
periodic point of $f$ and the restriction of $df^{\mathrm{per}(P)}(P)$ have two contracting complex eigenvalues.
If $H(P)$ is non-trivial, then $C^1$-arbitrarily close to $f$ one can find a $C^1$-diffeomorphism $g$
such that following holds: There exists a volume-expanding 
hyperbolic periodic point $Q$ whose index is $1$ such that
$P(g)$ and $Q$ are homoclinically related, $dg^{\mathrm{per}(Q)}(Q)$ has only positive and real eigenvalues,
and the restriction of $dg^{\mathrm{per}(Q) }(Q)$ to 
the stable direction is a homothety.
\end{lemma}

Let us prove Proposition \ref{p1degtan} assuming Lemma \ref{lcomp} and \ref{lhomo}.
\begin{proof}[Proof of Proposition \ref{p1degtan}]
Suppose $f$, $P$ are given as is in the hypothesis of Proposition \ref{p1degtan}.
First, by Lemma \ref{lcomp} we get arbitrarily close to $f$ we can find 
$f_1$ such that there exists a volume-expanding hyperbolic periodic point $Q(f_1)$ whose index is $1$ such that
$P(f_1)$ and $Q(f_1)$ are homoclinically related, and the restriction of $df_1^{\mathrm{per}(Q(f_1)) } (Q(f_1))$ 
has two complex eigenvalues.
Second, by applying Lemma \ref{lhomo} to $f_1$ and $Q(f_1)$, we can find $f_2$ in any neighborhood of $f_1$ 
such that there exists a volume-expanding hyperbolic periodic point $R(f_2)$ whose index is $1$,
$Q(f_2)$ and $R(f_2)$ are homoclinically related and the restriction of $df_2^{\mathrm{per}(R(f_2))}(R(f_2)) $ 
to the stable direction is a homothety, and a positive eigenvalue for unstable direction.
Note that if we take $f_2$ sufficiently close to $f_1$, then $P(f_2)$ and $Q(f_2)$ remain homoclinically related 
and thus $P(f_2)$ and $R(f_2)$ are homoclinically related, too. 
Since $f_1$ can be constructed arbitrarily close to $f$ and
$f_2$ can be constructed arbitrarily close to $f_1$,
$f_2$ can be constructed arbitrarily close to $f$. This ends the proof of our proposition.
\end{proof}

In subsection \ref{ssprf-cmp}, we prepare some techniques for the proof of 
Lemma \ref{lcomp}. 
In subsection \ref{ssprf-c-h}, we prove Lemma \ref{lhomo}
and in subsection \ref{ssprf-path} 
we give the proof of Lemma \ref{lpath} that is needed to prove Lemma \ref{lcomp}

\subsection{Proof of Lemma \ref{lcomp}}\label{ssprf-cmp}
In this subsection, we prove Lemma \ref{lcomp} assuming some results.
First, we collect some results for the the proof of Lemma \ref{lcomp}.  

The first one is the Franks' lemma (see appendix A of \cite{BDV}).
\begin{lemma}[Franks' lemma]\label{Franks}
Let $f$ be a $C^1$-diffeomorphism defined on a closed manifold $M$
and consider any $\delta > 0$. Then there is $\varepsilon > 0$
such that, given any finite set $\Sigma \subset M$, any neighborhood $U$ of $\Sigma$,
and any linear maps $A_x : T_xM \to T_{f(x)}M \,  (x \in \Sigma)$ such that 
$A_x$ is $\varepsilon$-close to $df(x)$, 
there exists a $C^1$-diffeomorphism $g$ that is $\delta$-close to $f$ in
the $C^1$-topology, coinciding with $f$ on
$M \setminus U$ and on $\Sigma$, and $dg(x) =A_x$ for all $x \in \Sigma$.
\end{lemma}

We introduce the Franks' lemma that preserves invariant manifolds \cite{GouF}.
To state it clearly, we prepare some notations.
For a hyperbolic periodic point $X$ of a diffeomorphism $f$, we consider 
the space of linear cocycles over $\mathcal{O}(X)$ 
(remember that $\mathcal{O}(X)$ is the orbit of $X$).
Let us denote this space as $C(X)$, i.e., $C(X)$ is the set of maps
$\sigma \colon TM|_{ \mathcal{O}(X)} \to TM|_{ \mathcal{O}(X)}$
such that for all $i \in \mathbb{Z}$, $\sigma(f^i(X)) = \sigma(f^i(X), \, \cdot \, )$ is a linear isomorphism
from $T_{f^{i}(X)}M$ to $T_{f^{i+1}(X)}M$.
By abuse of notation, we denote the cocycle 
given by the restriction of $df$ to $TM|_{\mathcal{O}(X)}$ by  $df$.

We define a metric on this space as follows.
For $\sigma_1, \sigma_2 \in  C(X)$, the distance betweeen $\sigma_1$ and  $\sigma_2$
(denoted by $\mathrm{dist}(\sigma_1, \sigma_2)$) is defined to be the following:
\[\,
\max \left\{ \max_{x \in \mathcal{O}(X)} ||\sigma_1(x) -\sigma_2(x)||, 
\,\, \max_{x \in \mathcal{O}(X)} ||(\sigma_1(x))^{-1} - (\sigma_2(x) )^{-1} || \, \right\}.
\] 
For $\sigma \in C(X)$ we denote by $\tilde{\sigma}$ the 
first return map of $\sigma$, i.e., the linear endomorphism of $T_XM$ given by
$\sigma(f^{\mathrm{per}(X)-1}(X)) \circ \cdots \circ \sigma(f(X)) \circ \sigma(X)$.
We define the eigenvalues of $\sigma$ by the eigenvalues of $\tilde{\sigma}$.
We say that $\sigma$ is hyperbolic if none of the eigenvalues of $\tilde{\sigma}$ has its absolute value equal to one.
Note that the set of the hyperbolic cocycles forms an open set in $C(X)$.
Let $\gamma(t)$ be a continuous path  in $C(X)$, i.e.,
$\gamma(t)$ is a continuous map from $[0,1]$ to $C(X)$. 
We define the diameter of $\gamma (t)$ 
(denoted by $\mathrm{diam}(\gamma(t)$) ) to be the number 
$\max_{\,0 \leq s,t \leq 1} \mathrm{dist}(\gamma(s), \gamma(t))$.

Let $U$ be a neighborhood of $\mathcal{O}(X)$.
For $x \in W^s(X) \cap (M \setminus U)$, 
let $\alpha (x)$ be the least number
such that $f^{\alpha(x)}(x) \in U$ holds.
We define the {\it stable manifold of $X$ outside $U$} (denoted $W^s_{\mathrm{loc} \setminus U}(X)$) to be the set of points
that never leave $U$ once they enter $U$, more precisely,
\[
W^s_{\mathrm{loc} \setminus U}(X) := \{ x \in W^s(X) \cap (M \setminus U) \, | \,  \forall n \geq \alpha(x), \, f^n(x) \in U\} .
\]
Let $g$ be a diffeomorphism so close to $f$ that we can define the continuation $X(g)$ of $X$ for $g$.
We say that {\it $g$ locally preserves the stable manifold of $f$ outside $U$ }if 
$W^s_{\mathrm{loc} \setminus U}(X,g) \supset W^s_{\mathrm{loc} \setminus U}(X,f)$.
Similarly, we can define the unstable manifold outside $U$, etc.

Now let us state the precise statement of the lemma.
\begin{lemma}[Gourmelon's Franks' lemma, \cite{GouF}]\label{lgou1}
Let $f$ be a $C^1$-diffeomorphism of $M$ and $X$ be a hyperbolic periodic point of $f$.
Suppose that there exists a continuous path $\{ \gamma(t) \mid 0 \leq t \leq 1\}$ in $C(X)$ satisfying the following:
\begin{enumerate}
\item For all $i \in \mathbb{Z}$, $\gamma(0)(f^i(X)) = df(f^i(X))$.
\item The diameter $\mathrm{diam}(\gamma(t))$ is less than $\varepsilon >0$.
\item For all $0 \leq t \leq 1$, $\widetilde{\gamma (t)}$ is hyperbolic (hence the dimensions of stable and unstable spaces are constant).
\end{enumerate}
Then, given neighborhood $U$ of $\mathcal{O}(X)$ 
there exists a $C^1$-diffeomorphism $g$ that is $\varepsilon$-$C^1$-close to $f$ satisfying the following properties:
\begin{enumerate}
\item For all $i$,  $f^i(X)=g^i(X)$.  Especially, $X$ is a periodic point for $g$.
\item As a linear cocycle over $C(X)$, $dg = \gamma(1)$. Especially, $X$ is a hyperbolic periodic point of $g$ and $\mathrm{ind}(X(f)) =\mathrm{ind}(X(g))$.
\item The support of $g$ is contained in $U$, i.e., $\{ x \in M \mid f(x) \neq g(x) \} \subset U$.
\item $g$ locally preserves the stable and unstable manifolds of $X$ outside $U$.
\end{enumerate}
\end{lemma}
\begin{proof}
Apply Theorem 2.1 in \cite{GouF} putting $I = \{ \dim M - \mathrm{ind}(X) \}$ and $J=\{ \mathrm{ind}(X) \}$.
\end{proof}

As a consequence of Lemma \ref{lgou1}, we can prove the following lemma.
\begin{lemma}\label{lgou2}
In addition to the hypotheses of Lemma \ref{lgou1}, suppose that the following property holds:
\begin{enumerate}
\setcounter{enumi}{3}
\item There is a hyperbolic periodic point $Y$ that is homoclinically related to $X$.
\end{enumerate}
Then, there is a (small) neighborhood $V$ of $\mathcal{O}(X)$ and a $C^1$-diffeomorphism $h$ 
$\varepsilon$-$C^1$-close to $f$ 
satisfying all the conclusions of Lemma \ref{lgou1} and the following property:
\begin{enumerate}
\setcounter{enumi}{4}
\item $X(h)$ is homoclinically related to $Y(h)$.
\end{enumerate}
\end{lemma}
Let us prove Lemma \ref{lgou2} by Lemma \ref{lgou1}.

\begin{proof}
We fix an open neighborhood $W$ of $\mathcal{O}(Y)$ that has 
empty intersection with $\mathcal{O}(X)$.
Take a point $a \in W^s(X) \ti W^u(Y)$. 
Since $a \in W^s(X)$, by replacing $a$ with $f^{n}(a)$ for some $n>0$ if necessary,
we can assume $a \not\in W$ for all $l \geq 0$.  
Furthermore, by shrinking $W$ if necessary and using the fact $a \in W^u(Y)$, 
we can assume the following condition holds: Let $l_0>0$ be the least number that satisfies $f^{-l_0}(x) \in W$.  
Then there exists a small neighborhood $D^u_Y(a)$ of $a$ in $W^u(Y)$ such that 
$f^{-l}(D^u_Y(a)) \subset W$ for all $l \geq l_0$ and 
$f^{-l}(D^u_Y(a)) \cap W = \emptyset$ for all $0\leq l <l_0$.   

Similarly, we take a point $b \in W^s(Y) \ti W^u(X)$ such that 
$b \not\in W$ and the following holds: 
Let $l_1>0$ be the least number that satisfies $f^{l_1}(b) \in W$.  
There exists a small neighborhood $D^s_Y(b)$ of $b$ in $W^s(Y)$ such that 
$f^{l}(D^s_Y(b)) \subset W$ for all $l \geq l_1$ and
$f^{l}(D^s_Y(b)) \cap W = \emptyset$ for all $0\leq l <l_1$.    

Similar argument gives us an open neighborhood $V$ of $\mathcal{O}(X)$ satisfying
the following conditions:
\begin{enumerate}
\item For all $n \geq 0$, $f^{-n}(D^u_Y(a)) \cap V = \emptyset$, especially $a \not\in V$. 
\item For all $n \geq 0$, $f^{n} (D^s_Y(b)) \cap V = \emptyset$, especially $b \not\in V$.
\item Let $k_0>0$ be the least number that satisfies $f^{k_0}(a) \in V$.  
Then there exists a small neighborhood $D^s_X(a)$ of $a$ in $W^s(X)$ such that 
$f^{k}(D^s_X(a)) \subset V$ for all $k \geq k_0$ and $f^{k}(D^s_Y(b)) \cap V = \emptyset$ for all $0 \leq k < k_0$.  
\item Let $k_1>0$ be the least number that satisfies $f^{-k_1}(b) \in V$.  
Then there exists a small neighborhood $D^u_X(b)$ of $b$ in $W^u(X)$ such that 
$f^{-k}(D^s_X(a)) \subset V$ for all $k \geq k_1$  and $f^{-k}(D^s_X(a)) \cap V = \emptyset$ for all $0 \leq  k < k_1$.  
\end{enumerate}
By applying Lemma \ref{lgou1}, we can construct an $\varepsilon$-$C^1$-close perturbation $h$ of $f$
that satisfies all the conclusions of Lemma \ref{lgou1}.
We show that $X(h)$ is homoclinically related to $Y(h)$.
To see this, first we check $a \in W^s(X,h) \ti W^u(Y,h)$. 
Since $f^{-n}(D^u_Y(a)) \cap V = \emptyset$ for all $n \geq 0$ and 
the support of $h$ is contained in $V$,  $D^u_Y(a)$ is contained in $W^u(Y,h)$.
Second, by the condition of $D^s_X(a)$, it is contained in the stable manifold
outside $V$. Thus it is contained in $W^s(X,h)$ and now we know $a\in W^s(X,h) \ti W^u(Y,h)$.
Similarly, we can see  $b\in W^s(Y,h) \ti W^u(X,h)$.
Thus the proof is completed.
\end{proof}

We collect some results about linear algebra on linear cocycles
from \cite{BDP}.

The first one roughly says, in dimension two,
the absence of the domination
implies the creation of complex eigenvalues.
The origin of this type of arguments can be found in \cite{M}.

\begin{lemma}\label{lpath}
For any $K>0$ and $\varepsilon >0$, the following holds: 
Let $(\Sigma, f, E, A)$ be a two dimensional diagonalizable periodic linear 
cocycle with positive eigenvalues. If $(\Sigma, f, E, A)$ is  bounded by $K$ and does not admit dominated splittings,
then there exists a periodic point $X$  and 
a path $\{ \gamma(t) \mid 0 \leq t \leq 1 \}$ in $C(X)$ 
such that the following conditions hold:
\begin{enumerate}
\item $\gamma (0) = A|_{\mathcal{O}(X)}$.
\item $\mathrm{diam}(\gamma (t)) < \varepsilon$. 
\item $\det \widetilde{\gamma(t)} $ is independent of $t$.
\item Let $\lambda_m(t) \leq \lambda_b(t)$ be the absolute value of 
the eigenvalues of $\widetilde{\gamma(t)}$.
Then for $s<t$ we have $\lambda_m(s) \leq \lambda_m(t)$ and
$\lambda_b(s) \geq \lambda_b(t)$.
\item $\widetilde{\gamma(1)} $ has two complex eigenvalues.
\end{enumerate}
\end{lemma}

We give the proof of this lemma in the next subsection.

The next lemma is used to create complex eigenvalues from a 
linear map which has eigenvalues with multiplicity two.
\begin{lemma}\label{koukousei}
Let $A$ be a linear endomorphism on a two-dimensional normed linear space
such that the eigenvalues of $A$ are 
positive real number $\lambda$ with multiplicity $2$.
Then there exists $B$ arbitrarily close to $A$ such that $B$ has two complex eigenvalues.
\end{lemma}

\begin{proof}
By taking the Jordan cannonical form of $A$, we can assume that $A$ has the following form:
\[
\begin{pmatrix}
\lambda & t \\
0            & \lambda
\end{pmatrix}.
\]
When $t=0$, the matrix
\[
\begin{pmatrix}
\cos \alpha            & - \sin \alpha   \\
\sin \alpha            &   \cos \alpha
\end{pmatrix}\begin{pmatrix}
\lambda & t \\
0            & \lambda
\end{pmatrix}
\]
has two complex eigenvalues for sufficiently small $\alpha >0$ and $\alpha$ can be taken arbitrarily small.
If $t\neq 0$, the matrix
\[
\begin{pmatrix}
\lambda & t \\
-\alpha /t            & \lambda
\end{pmatrix}.
\]
has complex eigenvalue for all $\alpha >0$. 
\end{proof}

The following lemmas enables us to``lift up'' the perturbation on some sub cocycle or 
quotient cocycle to the perturbation on the original cocycle. 
\begin{lemma}\label{llift}
Given $\varepsilon > 0$, a linear cocycle $(\Sigma, f, E, A)$ and
its invariant subcocycle $(\Sigma, f, F, A|_F)$, where $A|_F$ denotes 
the restriction of $A$ to $F$, the following hold:
\begin{enumerate}
\item Let $(\Sigma, f, F, B)$ be a cocycle with $\mathrm{dist}(A|_F, B) < \varepsilon$.
Then there exists a linear cocycle $(\Sigma, f, E, \tilde{B})$ such that 
$\mathrm{dist}(A, \tilde{B}) < \varepsilon$ and $A/F = \tilde{B}/F$,
where $A/F$ denotes the bundle map derives from $A$ on $E/F$.
\item If $(\Sigma, f, E/F, B)$ is a cocycle with $\mathrm{dist}(A/F, B) < \varepsilon$.
Then there exists a linear cocycle $(\Sigma, f, E, \hat{B})$ such that
$\mathrm{dist}(A, \hat{B}) < \varepsilon$, $\hat{B}$ leaves $F$ invariant 
and $A|_F = \hat{B}|_F$.
\end{enumerate}
\end{lemma}
\begin{proof}
See Lemma 4.1 in \cite{BDP}.
\end{proof}

\begin{lemma}\label{lind}
Let $E_1 \oplus E_2 \oplus E_3$ be a splitting of a linear cocycle.
If $E_1$ is not dominated by $E_2 \oplus E_3$, then 
one of the following holds:
\begin{enumerate}
\item $E_1$ is not dominated by $E_2$.
\item $E_1/E_2$ is not dominated by $E_3/E_2$.
\end{enumerate}
\end{lemma}

\begin{proof}
See Lemma 4.4 in \cite{BDP}.
\end{proof}

We prepare some lemmas which enables us to reduce our problem to specific sets. 
For a homoclinic class $H(P)$, by $\mathrm{per}_{+, \mathbb{R}}(H(P))$ 
we denote  the set of the 
volume-expanding hyperbolic periodic points that have the differentials with distinct positive real eigenvalues 
and are homolcinically related to $P$.

\begin{lemma}\label{ldense1}
Given $f \in \mathrm{Diff}^1(M)$ and a volume-expanding hyperbolic periodic point $P$,
if $H(P)$ is non-trivial, one can find $g\in \mathrm{Diff}^1(M)$ arbitrarily $C^1$-close to $f$ such that
$\mathrm{per}_{+,\mathbb{R}}(H(P,g))$ is dense in $H(P,g)$.
\end{lemma}
\begin{proof} 
See Proposition 2.3 in \cite{ABCDW} and Remark 4.17 in \cite{BDP}.
\end{proof}

\begin{lemma}\label{ldense2}
Let $f \in \mathrm{Diff}^1(M)$ and $P$ be a hyperbolic periodic point of $f$. 
Suppose $H(P)$ does not admit dominated splittings.
Then any dense $f$-invariant subset $\Sigma \subset H(P)$ 
does not admit dominated splittings.
\end{lemma}
\begin{proof} 
See Lemma 1.4 in \cite{BDP}.
\end{proof}

Let us start from the proof of Lemma \ref{lcomp}.
\begin{proof}[Proof of Lemma \ref{lcomp}]
By the quotations we prepared, large part of our proof is already finished. 
Let us see how to combine them to give the proof.

Suppose that $f$ and $P$ are given as in the hypothethis of Lemma \ref{lcomp}.
Lemma \ref{ldense1} implies there exists $f_1$ arbitrarily $C^1$-close to $f$ such that 
$\mathrm{per}_{+,\mathbb{R}}(H(P,f_1))$ is dense in $H(P,f_1)$.
Let us consider the periodic linear cocycle derives from $df_1$ 
on $\mathrm{per}_{+,\mathbb{R}}(H(P,f_1))$.
Since each periodic point has only real and positive  distinct eigenvalues, 
we can find a splitting $E_1 \oplus E_2 \oplus E_3$ on this cocycle 
such that corresponidng eigenvalues are in the increasing order at each point.
Note that since $H(P,f_1)$ does not admit dominated splitting, 
by Lemma \ref{ldense2},
$E_1$ is not dominated by $(E_2 \oplus E_3)$.
Then Lemma \ref{lind} says that either $E_1$ is not dominated by $E_2$ 
or $E_1 / E_2$ is not dominated by $E_3/E_2$.
We show that we can create the periodic point we claimed in both cases.

Let us consider the first case, where $E_1$ is not dominated by $E_2$.
We fix $\varepsilon >0$.
Since $M$ is compact, the linear cocycle $df_1$ restricted to $E_1 \oplus E_2$ is bounded. 
By applying Lemma \ref{lpath} to this cocycle, 
we get a periodic point $Q \in \mathrm{per}_{+,\mathbb{R}}(H(P,f_1))$ and a path $\gamma(t)$
satisfying the conclusions in Lemma \ref{lpath}.  
We can assume $Q \neq P$ by letting $\varepsilon$ sufficiently small.
Since $\lambda_m(t)\leq \lambda_b(t) \leq \lambda_b(0) < 1$,
there is no index bifurcation during the perturbation.

Then Lemma \ref{llift} tells us there exists a path $\Gamma(t) \subset C(Q)$ 
such that $X=Q$, $Y=P$ and $\Gamma(t)$ satisfies all the hypotheses of Lemma \ref{lgou2}. 
Hence by applying Lemma \ref{lgou2} we get a  $C^1$-diffeomorphism $f_2$
that is $\varepsilon$-$C^1$-close to $f_1$ such that $P(f_2)$ and $Q(f_2)$
satisfy all the conditions we need. Since $f_1$ can be arbitrarily close to $f$
and $\varepsilon$ can be taken arbitrarily small, the proof is completed in this case. 

Let us see the case where $E_1 / E_2$ is not dominated by $E_3/E_2$. 
Take $\varepsilon >0$.
Since $df$ is bounded, the cocycle induced on the quotient bundle 
$E_1 / E_2 \oplus E_3/E_2$ is also bounded. 
By applying Lemma \ref{lpath} to this cocycle, 
we obtain a periodic point $Q \in \mathrm{per}_{+,\mathbb{R}}(H(P,f_1))$, 
a path $\gamma(t)$  
with $\mathrm{diam}(\gamma(t)) < \varepsilon$ such that they satifsy 
all conclusions of the Lemma \ref{lpath}.  We claim that there is $t_0 \in (0,1)$ such that $\lambda_m(t_0)$
(not bigger eigenvalue of $\widetilde{\gamma (t_0)}$) is equal to the 
eigenvalue of $df_1^{\mathrm{per}(Q)}|_{E_2(Q)}$.
Let us denote the eigenvalues of $df_1^{\mathrm{per}(Q)}(Q)$ for the $E_i$-direction by $\mu_i$ ($i=1, 2, 3$).
Since $Q$ is volume-expanding,  $\mu_1 \mu_2 \mu_3 >1$, 
and since $Q$ has index $1$, $\mu_2 <1$. 
Hence, we have $\mu_1 \mu_3 > \mu_1 \mu_2 \mu_3 >1$.

Note that $\lambda_m(0) = \mu_1 < \mu_2$. 
When $t=1$, by Lemma \ref{lpath}, we have $\lambda_m(1) = \lambda_b(1)$.
Since $\det \widetilde{\gamma(t)}$ is independent of $t$ and $\mu_1 \mu_3 >1$,
we get $1< \mu_1 \mu_3 = \lambda_m(1) \lambda_b(1) = \left({}\lambda_m(1)\right)^2$. 
Then we can see $\lambda_m(1)>1$ because $\lambda_m(1)$ is positive. 
Finally, by $\lambda_m(0) = \mu_1 < \mu_2 <1 < \lambda_m(1)$ and the continuity of $\gamma(t)$, 
there is $t_0 \in (0, 1)$ such that $\lambda_m(t_0) =\mu_2$.
Then, we redefine $\gamma(t)$ as follows: $\gamma(t)$ is equal to the original $\gamma(t)$
when $t \in [0, t_0]$, otherwise $\gamma(t) = \gamma(t_0)$.
For this modified path $\gamma(t)$, we have
$\lambda_b(t) \geq \lambda_b(t_0) = \mu_1 \mu_3 / \lambda_m(t_0)  = \mu_1 \mu_3 / \mu_2  > \mu_1 \mu_3 >1$  for all $t$.

By Lemma \ref{llift} we can find a path $\Gamma(t) \subset C(X)$
such that $X=Q$, $Y=P$ and $\Gamma(t)$ satisfy all the hypotheses of Lemma $\ref{lgou2}$.
Applying Lemma $\ref{lgou2}$, we obtain a $C^1$-diffeomorphism $f_2$
that is $\varepsilon$-$C^1$-close to $f_1$ such that
$Q(f_2)$ are homoclinically related to $P(f_2)$ and
the differential $df_2^{\mathrm{per}(Q(f_2))}(Q(f_2))$ 
restricted to the stable direction of $T_{Q(f_2)}M$
has the eigenvalue $\mu_2$ with multiplicity $2$. Now Lemma \ref{koukousei} 
gives $f_3$ which is arbitrarily $C^1$-close to $f_2$ such that $P(f_3)$ and $Q(f_3)$ satisfy all the conditions we need.
\end{proof}

\begin{remark}
We point out two mistakable arguments in the proof of second case. 
\begin{enumerate}
\item The argument of Lemma \ref{koukousei} is necessary.
In general, $df_2^{\mathrm{per}(Q(f_2))}(Q(f_2))$ restricted to the stable direction of $T_{Q(f_2)}M$
is not diagonalizable. We only know there are two eigenvectors one in $E_2$ and the other in $E_1/E_2$.
These facts do not guarantee that we have two linearly independent eigenvectors in $E_1 \oplus E_2$.
\item One may wonder why we can get homoclinic relation between $P(f_3)$ and $Q(f_3)$
just by assuring the closeness of $f_2$ and $f_3$. It is because we fix the periodic points
and never change till the end. On the contrary, 
for example, in the proof of Lemma \ref{lpath}, 
the situation is more subtle because
we need to change the periodic point we choose to decrease the size of the perturbation.
\end{enumerate}
\end{remark}

\subsection{Proof of Lemma \ref{lhomo}}\label{ssprf-c-h}

Here we give the proof of Lemma \ref{lhomo}.
Before going into the detail, we see the naive idea of the proof.
We start from a non-trivial homoclinic class $H(P)$ such that $\mathrm{ind}(P)=1$ 
and the differential restricted to the stable direction has two complex eigenvalues.  
First, by a small perturbation, 
we create a hyperbolic periodic point $Q$ which is homoclinically related to $P$
and the differential at $T_QM$ has positive and distinct eigenvalues 
(this step is carried out inside the proof of Lemma \ref{lmix}).
Then we pick up a periodic point $R_n$ such that 
the differential to the stable direction is  ``mixed up'' under the influence of dynamics
around $P$. 
Now we perturb the diffeomorphism along the orbit of $R_n$ 
using the Franks' lemma such that 
the restriction of differential at $R_n$ to the stable direction is the homothety.
We pick up $R_n$ sufficiently close to $Q$ so that the resulted periodic point
has the homoclinic relation with $Q$.

To demonstrate this naive idea rigorously and clearly, 
we need the techniques developed in \cite{BDP} about 
linear cocycles admitting {\it transitions}.
We do not give the detailed review of \cite{BDP} here.
Instead, we give some explanation about how the techniques
are used in our argument so that the reader who is not 
well acquainted with that technique can understand 
what we do.

In the statement of Lemma \ref{lmix}, we refer to section 1 of \cite{BDP} 
for the definition and fundamental properties of linear cocycle with transition.
Roughly speaking, the transition matrix from a hyperbolic periodic point $Q$ to itself is a matrix of the 
differential of the ``return map'' from a neighborhood of $Q$ to itself. See Remark \ref{rtrans}.
\begin{lemma}[Lemma 5.4 in \cite{BDP}]\label{lmix}
Let $(\Sigma, f, E,A)$ be a continuous three-dimensional periodic linear cocycle 
with transitions and $\varepsilon_0$ be some positive real number. 
Assume that there exists $X \in \Sigma$ such that $\mathrm{ind}(X)=1$ and
$A^{\mathrm{per}(X)}(X)$ has two contracting complex eigenvalues. 
Then for every $0 < \varepsilon_1 < \varepsilon_0$ there is a point $Q \in \Sigma$ 
and an $\varepsilon_1$-transition $[t]$ from $Q$ to itself with the following properties:
\begin{enumerate}
\item There is an $\varepsilon_1$-perturbation 
$\tilde{A}^{\mathrm{per}(Q)}(Q)$ of $A^{\mathrm{per}(Q)}(Q)$ 
such that the corresponding matrix has only real positive eigenvalues with multiplicity one.
We denote the eigenspaces of the matrix by $E_1$, $E_2$ and $E_3$
so that the corresponding eigenvalues are in the increasing order. 
\item There is an $(\varepsilon_0 + \varepsilon_1)$-perturbation $[\tilde{t}]$ of 
the transition $[t]$ from $Q$ to itself such that the corresponding matrix $\tilde{T}$ satisfies the following:
\begin{itemize}
\item $\tilde{T}(E_3) = E_3$.
\item $\tilde{T}(E_1) = E_{2}$ and $\tilde{T}(E_{2}) = E_1$.
\end{itemize}
\end{enumerate}
Furthermore, if there exists $Y \in \Sigma$ such that $J(A^{\mathrm{per}(Y)}(Y))$ is bigger than $1$,
then we can choose the point $Q$ and the perturbation $\tilde{A}$ in the lemma
such that $J(\tilde{A}^{\mathrm{per}(Q)}(Q) )>1$.
\end{lemma}

\begin{remark}\label{rtrans}
The existence of the transition from $Q$ to itself tells us that
given any matrix that is obtained as the product of some $\tilde{A}^{\mathrm{per}(Q)}(Q)$ and some $\tilde{T}$, 
we can find a periodic orbit  whose differential is close to that matrix.
In the proof, we use the existence of transition to find the periodic point $R_n$
whose differential is close to a matrix which we want to create.
\end{remark}

Let us begin the proof of Lemma \ref{lhomo}.
\begin{proof}[Proof of Lemma \ref{lhomo}]

Given $f$ and $P$ as is in the hypothesis of the Lemma \ref{lhomo},
we fix $\varepsilon >0$ and a point $P' \in W^s(P) \ti W^u(P) \setminus \mathcal{O}(P)$
(we can take such $P'$ because $H(P)$ is non-trivial)
and fix a neighborhood $V$ of $\mathcal{O}(P) \cup \mathcal{O}(P')$ 
satisfying the following: For every $g$ that is $\varepsilon$-$C^1$-close to $f$,
if $Z$ is a periodic point such that $\mathcal{O}(Z,g)$
is contained in $V$, then $\mathrm{ind}(Z) =  \mathrm{ind}(P)$ and 
$P$ and $Z$ are homoclinically related. We can take such $V$ because
the set $\mathcal{O}(P) \cup \mathcal{O}(P')$ is uniformly hyperbolic.

Let us take a uniformly hyperbolic set $\Sigma_1$ contained in 
$V$ and contains $P$ and $P'$ such that $\Sigma_1$ admits transitions
(in practice, $\Sigma_1$ is 
a generalized horseshoe containing $P$ and $P'$ as $\Sigma_1$).
We fix a basis of tangent space of each point in $\Sigma_1$ and we identify
the differential maps between tangent spaces with some matrices.
For the detail, see section 1 of \cite{BDP}.

We denote the set of hyperbolic periodic points in $\Sigma_1$ by $\Sigma_2$.
Let us apply Lemma \ref{lmix} to the periodic linear cocycle 
$(\Sigma_2, f, TM|_{\Sigma_2},df)$ with $\varepsilon_0= \varepsilon/2$ and 
$\varepsilon_1 = \varepsilon/4$.
Then we get a periodic point $Q \in \Sigma_2$ such that 
the following property holds for $df \in C(Q)$
(remember that $C(Q)$ denotes the set of cocycles on $Q$):
$df$ is $\varepsilon_1$-close to a cocycle $\sigma \in C(Q) $ whose first return map $\widetilde{\sigma}$
has only real, positive and distinct eigenvalues. 
Note that we can choose $Q$ so that $\det \widetilde{\sigma} > 1$
since $P$ is volume-expanding.
We denote the eigenspaces of $\widetilde{\sigma}$ by $E_1$, $E_2$ and $E_3$
so that the corresponding eigenvalues are in the increasing order. 
We also know the transition with matrix $T$ from $Q$ to itself
has the property as is described in Lemma \ref{lmix},
more precisely, there exists a matrix $\tilde{T}$ that is $(\varepsilon_0 + \varepsilon_1)$-close to $T$
such that the following holds:
\begin{itemize}
\item $\tilde{T}(E_3) = E_3$.
\item $\tilde{T}(E_1) = E_{2}$ and $\tilde{T}(E_{2}) = E_1$.
\end{itemize}

Let us consider a matrix given as follows:
\[
D_n := \widetilde{\sigma}^{2n} \circ \tilde{T} \circ \widetilde{\sigma}^{n} \circ \tilde{T} \circ
 \widetilde{\sigma}^n \circ \tilde{T} \circ  \widetilde{\sigma}^{2n} \circ \tilde{T}.
\]
Then, since $Q$ admits the transition 
with matrix $T$, there exists a periodic point $R_n$ 
such that the cocycle $df \in C(R_n)$ is $\varepsilon_1 + (\varepsilon_0 +\varepsilon_1) =\varepsilon$-close to the 
cocycle $\tau$ such that $\widetilde{\tau}$ is given by $D_n$.

Now, by applying the Franks' lemma (see Lemma \ref{Franks}), 
we get a diffeomorphism $f_n$ that is $\varepsilon$-$C^1$-close to $f$ such that
$R_n$ is a hyperbolic periodic point of $f_n$ and the differential
$df_n^{\mathrm{per}(R_n)}(R_n)$ is equal to $D_n$.

We show that, for each $n$, the linear map $df_n^{\mathrm{per}(R_n)}(R_n)$
leaves $E_1 \oplus E_2$ invariant and acts  as a homothetic endomorphism,
leaves $E_3$ invariant and the eigenvalue of $E_3$ direction is positive.
To see this, let us denote the eigenvalue 
of $\widetilde{\sigma}$ for the $E_i$-direction 
by $\lambda_i$ ($i=1, 2, 3$),
where $\lambda_i$ is some positive real number and
put $\tilde{T}(e_1) = \mu_1 e_2$,  $\tilde{T}(e_2) = \mu_2 e_1$ and $\tilde{T}(e_3) = \mu_3 e_3$,
where $\mu_1$, $\mu_2$ and $\mu_3$ are non-zero real numbers.
Then direct calculations show that $df_n^{\mathrm{per}(R_n)}(R_n)(e_i) = (\mu_1 \mu_2)^2(\lambda_1 \lambda_2)^{3n}e_i$ for $i=1, 2$
and $df_n^{\mathrm{per}(R_n)}(R_n)(e_3) = \mu_3^4 \lambda_3^{6n}e_3$.
If $n$ is sufficiently large, then $\det(df_n^{\mathrm{per}(R_n)}(R_n))$ is greater than one since $\det \widetilde{\sigma}$ 
is greater than one. 

Hence we finished the proof.
\end{proof}

\begin{remark}
\normalfont
The form of the matrix $D_n$ may look bizarre.  
Let us see why we need to pick up this matrix.
We want to choose the matrix whose restriction is a homothety to the stable direction.
For instance, the matrix $\widetilde{\sigma}^n \circ \tilde{T} \circ \widetilde{\sigma}^n \circ \tilde{T}$
satisfies this condition. However, this matrix is a power of the matrix  $\widetilde{\sigma}^n \circ \tilde{T}$.
Hence in general this matrix is not approximated by a first return map of the differential of the some periodic point
(see the definition of the transition).
We choose $D_n$ so that it is not a power of some matrix.
\end{remark}

\subsection{Proof of Lemma \ref{lpath}} \label{ssprf-path}
Finally, we give the proof of Lemma \ref{lpath}.
The argument of the proof is similar to that of 
Proposition 3.1 in \cite{BDP}. 
In addition to the original proof, we need to check two things. 
The first one is that during the perturbation there
is no index bifurcation. The second one is the perturbation
is uniformly small.

We divide the proof of Lemma \ref{lpath} into three steps.
The first step (Lemma \ref{lb-c}) tells us if there is a periodic point 
at which eigenspaces 
forms a small angle, then one can create complex eigenvalues by a small perturbation.
In the second step (Lemma \ref{ld-b}), we prove that if a periodic linear system 
does not admit dominated splittings, 
then one can construct a periodic point with a small angle
by a small perturbation.
Finally, we combine previous two techniques to prove Lemma \ref{lpath}.

We prepare some notations.
Given two one-dimensional subspaces $V_1, V_2$ in a two-dimensional Euclidian space, 
the angle between $V_1, V_2$ is the unique real number $0 \leq \alpha \leq \pi/2$
that satisfies $\cos \alpha = |(v_1, v_2)|/(|v_1||v_2|)$, where
$v_i$ is any non-zero vector in $V_i$ ($i=1, 2$), $(\,\cdot\,, \,\cdot\,)$ is the inner product
and $|\,\cdot\,|$ is the norm derives from the inner product.

Let us state the first and the second steps.
\begin{lemma}\label{lb-c}
For any $K>0$ and $\varepsilon >0$, there exists 
$\alpha = \alpha(\varepsilon, K) >0 $ such that the following holds:
Let $(\Sigma, f, E, A)$ a two-dimensional diagonalizable periodic linear 
cocycle with real, positive distinct eigenvalues, bounded by $K$. 
Suppose there is a point $X$ 
at which the angles between two eigenspaces are less than $\alpha$.
Then there exists a path $\{ \gamma(t) \mid 0 \leq t \leq 1 \}$ in $C(X)$ 
such that the following conditions hold:
\begin{enumerate}
\item $\gamma (0) = A|_{\mathcal{O}(X)}$.
\item $\mathrm{diam}(\gamma (t)) < \varepsilon$. 
\item $\det \widetilde{\gamma(t)} $ is independent of $t$.
\item Let $\lambda_m(t) \leq \lambda_b(t)$ be the absolute value of 
the eigenvalues of $\widetilde{\gamma(t)}$.
Then for $s<t$ we have $\lambda_m(s) \leq \lambda_m(t)$ and
$\lambda_b(s) \geq \lambda_b(t)$.
\item $\widetilde{\gamma(1)} $ has two complex eigenvalues.
\end{enumerate}
\end{lemma}

\begin{lemma}\label{ld-b}
For any $K>0$, $\varepsilon >0$ and $\alpha > 0$, the following holds:
Let $(\Sigma, f, E, A)$  be a two-dimensional diagonalizable periodic linear 
cocycle with real positive distinct eigenvalues, bounded by $K$ and does not admit dominated splittings. 
Then there exists a hyperbolic periodic point $X \in \Sigma$  and 
a path $\{ \gamma(t) \mid 0 \leq t \leq 1 \}$ in $C(X)$ 
such that the following conditions hold:

\begin{enumerate}
\item $\gamma (0) = A|_{\mathcal{O}(X)}$.
\item $\mathrm{diam}(\gamma (t)) < \varepsilon$. 
\item $\det \widetilde{\gamma(t)} $ is independent of $t$.
\item Let $\lambda_m(t) \leq \lambda_b(t)$ be the absolute values of 
the eigenvalues of $\widetilde{\gamma(t)}$.
Then for $s<t$ we have $\lambda_m(s) \leq \lambda_m(t)$ and
$\lambda_b(s) \geq \lambda_b(t)$.
\item $\widetilde{\gamma (1)}$ has two eigenspaces 
with angle less than $\alpha$.
\end{enumerate}
\end{lemma}

We need an auxiliary lemma about the effect of 
a perturbation on the constant of the boundedness of a cocycle.

\begin{lemma}\label{bdd}
Let $\sigma$ be a cocycle in $C(X)$ where $X$ is a periodic point, 
and let $\gamma(t)$ be 
a path with $\mathrm{diam}(\gamma (t)) < \varepsilon$. 
Then $\gamma(1)$ is bounded by $\varepsilon + \| \sigma \|$.
\end{lemma}

\begin{proof}
For every unit vector $v$ in some $T_{f^i(X)}M$,
we have
\begin{align*}
  &\| (\gamma(1) (f^i(X)))(v) \|  \\
\leq& \| (\gamma(1) (f^i(X)))(v) - (\gamma(0) (f^i(X)))(v) \|  +\| (\gamma(0) (f^i(X)))(v) \| 
\leq \varepsilon + \| \sigma \|.
\end{align*}
We can prove a similar inequality for the inverse of $\gamma(1)$.
\end{proof}

We give the proof of Lemma \ref{lpath} using these three lemmas.
\begin{proof}[Proof of Lemma \ref{lpath}]
Let $(\Sigma, f, E, A)$ be a two-dimensional diagonalizable periodic linear 
cocycle with positive distinct eigenvalues, bounded by $K$ and
suppose that it does not admit dominated splittings.

First, let us assume the angles between two eigenspaces are not bounded below, 
more precisely,
there exists a sequence of periodic points $(X_n)$ such that 
the sequence of the angles $(\alpha_n)$, where $\alpha_n$ is the angle of the 
eigenspaces of $A^{\mathrm{per}(X_n)}(X_n)$, converges to zero.
Fix $\varepsilon >0$ and the constant $\alpha_0 = \alpha(K, \varepsilon)$ in Lemma \ref{lb-c}.
Since $\alpha_n \to 0$ as $n \to \infty$, we can take $n_0$ such that $\alpha_{n_0} < \alpha_0$.
Then Lemma \ref{lb-c} enable us to find a path $\gamma(t)$ in $C(X)$ 
as we claimed.

Hence, we can assume that the angles between two eigenspaces point are uniformly bounded at any point.
By taking appropriate basis at each point, 
we can assume that 
the corresponding eigenspaces are orthogonal at each periodic point. 
Fix $\varepsilon >0$ and the constant $\alpha_1 = \alpha(K+\varepsilon/2, \varepsilon/2)$ in Lemma \ref{lb-c}. 
Applying Lemma \ref{ld-b} to  $(\Sigma, f, E, A)$ with constants $K$, $\varepsilon/2$, and $\alpha_1$,
we get a periodic point $X \in \Sigma$ and a path $\gamma_1(t)$ in $C(X)$
satisfying all the properties in the conclusion of Lemma $\ref{ld-b}$
for $\gamma_1(t)$, $\varepsilon/2$ and $\alpha_1$.
Then Lemma \ref{bdd} implies that the cocycle $\gamma_1(1)$ is bounded by $K + \varepsilon/2$.
Since the angle between eigenspaces of $\widetilde{\gamma_1(1)}$ is less than $\alpha_1$, 
we can apply Lemma \ref{lb-c} 
(replacing  $K$, $\alpha$ and $\varepsilon$ with
$K + \varepsilon/2$, $\alpha_1$ and $\varepsilon/2$ respectively) in order to 
obtain a path $\gamma_2(t)$.

Now we construct a continuous path $\gamma(t)$ as follows:
If $0 \leq t \leq 1/2$ then  $\gamma(t) = \gamma_1(2t)$.
If $1/2 \leq t \leq 1$ then  $\gamma(t) = \gamma_2(2t-1)$.
We have $\mathrm{diam}(\gamma(t))$ is less than $\varepsilon $
because $\mathrm{diam}(\gamma_1(t))$, $\mathrm{diam}(\gamma_2(t))$ are less than  $\varepsilon /2$.
Thus we constructed the path we claimed and the proof is completed.
\end{proof}

Let us give the proof of Lemma \ref{lb-c}.
For the proof, we need an elementary lemma.
\begin{lemma}\label{eleine}
Let $R(x)$ denote the rotation matrix of angle $x$ on a two-dimentional Euclidian space. 
Then there exists a positive constant $C$
such that the following inequality holds:
\[
\| R(s) - R(t) \|  \leq C |s-t| \,\, \mbox{for all} \,\, -\pi /4 \leq s, t \leq \pi /4.
\]
\end{lemma}
\begin{proof}
We introduce a norm $\| \, \cdot \, \|_2$ on the space of linear maps as follows.
Fix an orthonormal basis of the Euclidean space and for a linear map $A$, define 
\[
\| A \|_2 =
\left\| \begin{pmatrix}
a & b \\
c & d
\end{pmatrix} \right\|_2 = \sqrt{a^2 +b ^2 +c^2 +d^2},
\]
where the matrix denotes the matrix representation of $A$
with respect to the basis.
We fix a constant $N$ such that for every $A$ the inequality $\| A \| \leq N \| A \|_2$ holds.
Then, 
\begin{align*}
& \quad \| R(s) - R(t) \|^2 \\
& \leq N^2 \| R(s) - R(t)\|_2^2 
= 2N^2 \big( (\cos(s)-\cos(t))^2 + (\sin(s)-\sin(t))^2 \big).
\end{align*}
Since $-\pi /4 \leq s, t \leq \pi /4$, we get
\[
 |\cos(s)-\cos(t)| \leq |s-t|/\sqrt2, \,  |\sin(s)-\sin(t)| \leq |s-t|.
\]
So we have
\[
2 N^2 \big( (\cos(s)-\cos(t))^2 + (\sin(s)-\sin(t))^2 \big) 
\leq 3 N^2 |s-t|^2.
\]
Hence for $C= \sqrt3N$ we have the inequality.
\end{proof}

\begin{proof}[Proof of Lemma \ref{lb-c}]
Let $X$ be a point at which the angle of corresponding eigenspaces is less than $\alpha$.
By taking an orthonormal basis, we can take a matrix representation of $A^{\mathrm{per}(X)}(X)$ with
\[
A^{\mathrm{per}(X)}(X)=
\begin{pmatrix}
\lambda_1  & \mu           \\
               0 & \lambda_2
\end{pmatrix},
\]
where $\lambda_1< \lambda_2$ is positive real numbers and $\mu$ is some non-zero real number.
By a direct calculation, we get $ (\lambda_2 -\lambda_1)/|\mu| < \tan \alpha$
(left hand side is the tangent of the angle between two eigenspaces at $X$). 

We fix $\varepsilon >0$ and put $\alpha = 
\alpha (K, \varepsilon) := \varepsilon / (CK)$, 
where $C$ is the constant given in Lemma \ref{eleine}.
Let us define a path $\gamma(t)$ as follows: 
\begin{align*}
[\gamma(t)](f^{i}(X)) =
\begin{cases}
 A|_{f^{i}(X)}, & \mbox{if} \,\,\, f^{i}(X) \neq f^{\mathrm{per}(X)-1}(X), \\
 R(-\mathrm{sign}(\mu )\alpha t)A|_{f^{i}(X)},  & \mbox{if} \,\,\, f^{i}(X) = f^{\mathrm{per}(X)-1}(X),
\end{cases}
\end{align*}
where $\mathrm{sign}(\mu)$ is the signature of $\mu$.

Let us see this path enjoys all the conditions we claimed in Lemma \ref{lb-c}.
In the following, we only treat the case when $\mu >0$. 
The proof for the case $\mu <0$ can be done similar way.
We first examine the diameter of this path. To see this, we only need to check the distance of
 $[\gamma(t)](f^{\mathrm{per}(X)-1})$, which can be estimated using Lemma \ref{eleine} as follows:
 \begin{align*}
    & \quad \,\, \| [\gamma(s)](f^{\mathrm{per}(X)-1}) - [\gamma(t)](f^{\mathrm{per}(X)-1}(X)) \|  \\
    & =  \| R(-\alpha s)A|_{f^{\mathrm{per}(X)-1}(X)} - R(-\alpha t)A|_{f^{\mathrm{per}(X)-1}(X)} \|  \\ 
   & \leq  \| R(-\alpha s) -  R(-\alpha t ) \| \|A|_{f^{\mathrm{per}(X)-1}(X)}\|  \\
   &\leq \alpha CK |s-t| \leq \varepsilon |s-t| \leq \varepsilon.
 \end{align*}
We can get similar estimates for the inverse.

The value $\det \widetilde{\gamma(t)}$ is independent of $t$, 
since $\gamma(t)$ is obtained by multiplying an orthogonal matrix to $\gamma(0)$.
Let us investigate the behavior of the eigenvalues. 
We denote by $\lambda_m(t), \lambda_b(t)$  ($\lambda_m(t) \leq \lambda_b(t)$) the eigenvalues of $\widetilde{\gamma}(t)$.
The characteristic equation of  $\widetilde{\gamma(t)}$ is given by
$x^2 -\theta x +d =0$, where $d=\lambda_1 \lambda_2$ and 
$\theta(t) = (\lambda_1 + \lambda_2) \cos (-\alpha t) + \mu \sin (-\alpha t)$.
So two eigenvalues are given as $(\theta \pm  \sqrt{\theta^2 -4d})/2$.
To complete the proof, it is enough to check the following properties hold:
\begin{itemize}
\item $\theta >0$ for all $0 \leq t \leq  1$.
\item $\theta$ is monotone decreasing when $t$ increases.
\item $\theta^2-4d <0$ when $t=1$.
\end{itemize}
Indeed, $\lambda_b(t) = $ is $(\theta +  \sqrt{\theta^2 -4d})/2$ when $\theta >0$.
Therefore if $\theta$ is monotone decreasing, so is $\lambda_b(t)$.
Moreover, it implies that
 $\lambda_m(t)$ is monotone increasing, 
since $\lambda_m(t)\lambda_b(t)$ is a positive constant.

Let us chek the items above.
First, observe that
\[
\theta(t) =\sqrt{(\lambda_1+\lambda_2)^2 + \mu^2} \sin(\beta -\alpha t),
\]
where $0< \beta < \pi/2$ is a real number satisfying $\tan \beta = (\lambda_1+\lambda_2)/ \mu$.
This shows that $\beta > \alpha$. Thus $\theta(t)$ is positive and monotone decreasing.

Finally, we see $\theta^2-4d <0$. When $t=1$, we obtain that
\begin{align*}
\theta(1) &= \frac{\mu(\lambda_1 + \lambda_2) - \mu(\lambda_2 - \lambda_1)}{\sqrt{(\lambda_2 - \lambda_1)^2 + \mu^2}} 
 = \frac{2\lambda_1}{\sqrt{1 + \left( \frac{\lambda_1 - \lambda_2}{\mu}\right)^2 } }
< 2 \sqrt{\lambda_1^2} <  2 \sqrt{\lambda_1\lambda_2} =2 \sqrt{d}.
\end{align*}
So, we finished the proof.
\end{proof}

Finally, let us give the proof of the Lemma \ref{ld-b}.
\begin{proof}[Proof of Lemma \ref{ld-b}]
The proof of Lemma \ref{ld-b} is just the repetition of Lemma 3.4 in \cite{BDP}.
We can easily check the behavior of the eigenvalues during the perturbation. 
So the proof is left to the reader.
\end{proof}

\section{Bifurcation of degenerate tangency}\label{biftan}

In this section we prove Proposition \ref{pbiftan}.

\subsection{Strategy of the proof of Proposition \ref{pbiftan}}

We divide the proof into two steps.
The first step is to reduce the dynamics into
affine dynamics.
The second step is to investigate the bifurcation of the dynamics.

To state our proof clearly, let us give the following definition.
\begin{definition}
Let $f \in \mathrm{Diff}^1(M)$ with $\dim M =3$ and $X$ be a hyperbolic 
index-1 fixed point with a homoclinic tangency and 
eigenvalues of $df(X)$ are positive and distinct.
A one-parameter family of the $C^1$-diffeomorphism $(f_t)_{|t| < \delta} \subset \mathrm{Diff}^1(M)$
is said to be an {\it affine unfolding of the degenerate tangency of $f$ with respect to $X$} if the following holds 
(see Figure \ref{fig:DTset}):
\begin{enumerate}
\def\labelenumi{($DT$\theenumi)}
\item $f_0 =f$.
\item There exists a coordinate chart $\phi : U \to \mathbb{R}^3$ around $X$ 
such that $\phi(U) = (-1,1)^3$ and $\phi(X)$ is the origin of $\mathbb{R}^3$.
\item  We put $F_t :=  \phi \circ f_t \circ \phi^{-1}$.
For  $(x, y, z) \in (-1,1) \times (-1 , 1) \times (-\mu^{-1}, \mu^{-1})$, 
we have $F_t (x, y, z) =(\lambda x, \tilde{\lambda} y , \mu z )$,
where $0< \tilde{\lambda} < \lambda < 1 <\mu$. 
We also require $X$, $\tilde{\lambda}$, $\lambda$ and $\mu$
are independent of $t$.
Thus locally $z$-axis is the unstable manifold of $\phi(X)$ and $xy$-plane is the 
stable manifold of $\phi(X)$.
\item There exist two points $P, Q \in U$ with $\phi(P) =(0, 0, p)$ and 
$\phi(Q) =(0, q, 0)$, where $0 < p, q < 1$, such that the following holds:
For some positive integer $N \geq 2$, $f_0^N(P) =Q$ and $f_t^{i}(P) \not\in U$ for $0 < i  < N$. This means $P, Q \in W^s(X) \cap W^u(X)$.
\item By the abuse of notation, 
we denote $\phi \circ f^N_t \circ \phi^{-1}$ by $F^N_t$.
Then, there exists a small neighborhood $W_0$ of $P$ with
$\phi(W_0) = [-\varepsilon,\varepsilon] \times [ -\varepsilon, \varepsilon]
\times [p- \varepsilon , p+ \varepsilon]$ and $F_t^N(W_0) \subset V$
such that  for every $(x, y, z +p) \in W_0$, $F_t^N(x, y, z +p) = (az, by +q, cx+t)$, 
where $a, b$ and $c$ are non-zero real numbers.
We put $W_1 := F_0^N(W_0)$ and call $F^N_t$ {\it return map}.
\end{enumerate}
\end{definition}

\begin{figure}[t]
\begin{center}
  \includegraphics[width=4in]{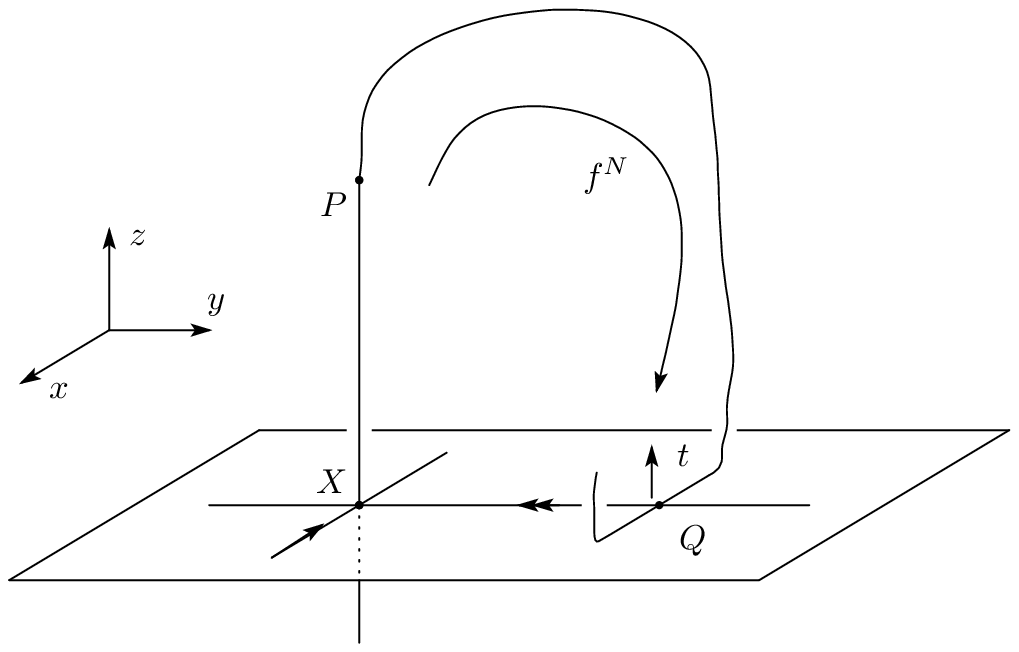}\\
 \caption{An affine unfolding of the degenerate tangency.}
 \label{fig:DTset}
  \end{center}
\end{figure}

Our first step is stated as follows:
\begin{proposition}\label{laffcst}
If $X$ has a homothetic tangency, then one can find a 
diffeomorphism $g$ $C^1$-arbitrarily close to $f$ such that
there exists a one-parameter family of $C^1$-diffeomorphisms 
$(g_t)_{|t| <\delta}$ that is an affine unfolding of the degenerate tangency 
of $g$ with respect to $X(g)$.
\end{proposition}

The second step is the following:
\begin{proposition}\label{laffbif}
Let $X$ be a volume-expanding 
hyperbolic periodic point with an affine unfolding of 
the degenerate tangency $(f_t)_{|t| <\delta}$ with respect to $X$.
Then, for arbitrarily small $\varepsilon >0$, there exists $0<\tau \leq \varepsilon$
such that $f_{\tau}$ has an index-two periodic point $Y$ and
there exists a heterodimensional cycle associated to $X(f_{\tau})$ and $Y$.
\end{proposition}
It is clear that these two propositions imply Propotision \ref{pbiftan}.

\subsection{Proof of Proposition \ref{laffcst}}
We give the proof of Proposition \ref{laffcst}.
For the proof, we prepare two lemmas.

The first one is a variant of the Franks' lemma. We omit the proof, 
since it is easily obtained from the proof of the Franks' lemma.

\begin{lemma}[Local linearization]\label{llin-map}
Let $f \in \mathrm{Diff}^1(M)$ with $\dim M = m$. Consider $x \in M$
and a coordinate neighborhoods $\phi : U \to \mathbb{R}^m$
and $\psi : V \to \mathbb{R}^m$ of $x$ and $f(x)$ respectively
such that $\phi(x)$ and $\psi(f(x))$ are the origin of  $\mathbb{R}^m$.
Then for any $\varepsilon >0$ and any neighborhood $U'$ of $x$, 
there exist a neighborhood $\tilde{U}$ of $x$ contained in $U'$ 
and $\tilde{f} \in \mathrm{Diff}^1(M)$ such that 
$\tilde{f}$ is $\varepsilon$-$C^1$-close to $f$,
$\tilde{f}$ coincides with $f$ on $M \setminus U'$ and
the map $(\psi \circ \tilde{f}\circ \phi^{-1})$ coincides with 
a linear map given by $d(\psi \circ f \circ\phi^{-1})(\mathbf{0})$ on $\tilde{U}$
(where $\mathbf{0}$ denotes the origin of $\mathbb{R}^m$).
\end{lemma}

The following lemma is a version of 
Gourmelon's Franks' lemma:
\begin{lemma}[Lemma 4.1 in \cite{GouF}]\label{llin-fix3}
Let $f \in \mathrm{Diff}^1(M)$ with $\dim M = m$.
Consider a periodic point  $x \in M$ and a coordinate neighborhood
$\phi : U \to \mathbb{R}^m$ of $x$ 
with $\phi(x) = \mathbf{0}$.
Then, for any $\varepsilon >0$ and any neighborhood $V \subset U$ of $x$, 
one can find a $C^1$-diffeomorphism $\tilde{f}$ that is
$\varepsilon$-$C^1$-close to $f$ 
and a small neighborhood $\tilde{V} \subset V$ such that 
$\tilde{f}(x) = f(x)$ for any $x \in M \setminus V$,
the restriction of $(\phi \circ\tilde{f}\circ\phi^{-1})$ to 
$\phi(\tilde{V})$ is equal to the linear map given by 
$d(\phi \circ f \circ\phi^{-1})(\mathbf{0})$
and $\tilde{f}$ locally preserves the invariant manifolds of $X$ outside $\tilde{V}$
(see Lemma \ref{lgou1} for definition).
\end{lemma}
Let us begin the proof of Proposition \ref{laffcst}.

\begin{proof}[Proof of Proposition \ref{laffcst}]
Let $f$ be given as is in the hypothesis of Proposition \ref{laffcst}. 
In the following, we assume that $X$ is a fixed point. 
The general cases can be reduced 
to this case by considering the power of $f$.

First, by using Lemma \ref{llin-fix3} in a sufficiently small neighborhood
and taking appropriate coordinate neighborhood, 
we can take $f_1$ that has the following properties:
\begin{itemize}
\item There exists a (smooth) 
coordinate neighborhood $\phi : U \to \mathbb{R}^3$ around $X$ 
such that $\phi(U) = (-1,1)^3$ and $\phi(X)$ is the origin of $\mathbb{R}^3$.
\item  We put $F_1 :=  \phi \circ f_1 \circ \phi^{-1}$.
For  $(x, y, z) \in (-1,1) \times (-1 , 1) \times (-\mu^{-1}, \mu^{-1})$, 
$F_1 (x, y, z) =(\lambda x, \lambda y , \mu z )$,
where $0<  \lambda < 1 <\mu$. 
\item There exist two points $P, Q \in U$ with $\phi(P) =(0, 0, p)$ and 
$\phi(Q) =(0, q, 0)$, where $0 < p, q < 1$ such that following holds:
For some positive integer $N \geq 2$, $f_1^N(Q) =P$ and $f^{i}(P) \not\in U$ for $0 < i  < N$. 
\item $W^s(X)$ and $W^u(X)$ are tangent at $Q$, in particular, 
$T_{Q}W^u(X)$ is contained in $T_{Q}W^s(X)$.
\end{itemize}

Let us give a perturbation so that the point of tangency is on the strong stable manifold of $X$.
First, we take an interval $J :=[a, b]$ that is contained in $(\mu^{-2}p, \mu^{-1}p)$
such that $J$ and the set $(\lambda^{n}q)_{n \geq 0}$ have the empty intersection.
Let $\rho(t)$ be a $C^{\infty}$-function on $\mathbb{R}$ such that the following holds:
\begin{itemize}
\item $\rho(t) =1$ if $|t| <a$.
\item $\rho(t) =0$ if $|t| >b$.
\end{itemize}
We modify $F_1$ to $F_2$ as follows:
\[
F_2(x,y,z) := (1-R(X))F_1(X) + R(X)(\lambda x, \tilde{\lambda} y, \mu z),
\]
where $R(x,y,z) := \rho(x)\rho(y)\rho(z)$ and $\tilde{\lambda}$ is a real 
number satisfying $0<\tilde{\lambda}< \lambda < 1$ 
and sufficiently close to $\lambda$.
Let us define the map $f_2$ as follows:
$f_2(x) = f_1(x)$ when $x \not\in U$.
Otherwise $f_2(x) =  \phi \circ F_2 \circ \phi^{-1}(x)$.
If we take $\tilde{\lambda}$ sufficiently close to $\lambda$,
then $f_2$ is a diffeomorphism of $M$. 
Note that $P$ and $Q$ are still 
the points of tangency of the invariant manifolds of $X$,
$P$ is on the strong stable manifold of $X$ and 
$f_2$ converges to $f_1$ when $\tilde{\lambda} \to \lambda$
in the $C^1$-topology. 
We fix $\tilde{\lambda}$ that is very close to $\lambda$ 
and give more perturbations.

Throughout the proof, we often change the coordinate in the following way:
Given a real number $r>1$, we define $r\mathrm{Id} : \mathbb{R}^3 \to \mathbb{R}^3 $
to be the expansion by multiplying $r$, more presicely, we put $r\mathrm{Id}(x, y, z) :=(rx, ry, rz)$.
Then $r\mathrm{Id}\circ \phi$ gives another coordinate neighborhood of $X$.
We replace $\phi$ with $r\mathrm{Id}\circ \phi$ and call this chart 
{\it the renormalization of $\phi$}.
When we take a renormalization, we change $V$, $P$ and $Q$ so that ($DT$2), ($DT$3), ($DT$4) hold. 
More precisely, we replace $V$ with $r^{-1}\mathrm{Id}(V)$, 
$P$ with $f^{-n_P}(P)$ and $Q$ with $f^{-n_Q}(Q)$,
where $n_P$ is the smallest number that 
satisfies $f^{-n_P}(P) \in r^{-1}\mathrm{Id}(V)$ and
$n_Q$ is the smallest number that satisfies $f^{n_Q}(Q) \in r^{-1}\mathrm{Id}(V)$.

Let us see the effect of the renormalization on the differential of the return map.
Given a diffeomorphism $f$ and a coordinate chart $\phi$, we have 
the differential of the return map $dF^N(P)$. 
If we take the renormalization, the differential of the return map is given by
$L^{n_Q} dF^N(P) L^{n_P}$, where $L$ is a matrix given as follows:
\[
L :=
\begin{pmatrix}
\lambda             & 0                                    & 0 \\
0                        & \tilde{\lambda}             & 0 \\
0                        & 0                                    & \mu 
\end{pmatrix}.
\]
Note that from this calculation we can see a component of $dF^N$ 
is equal to zero if and only if the corresponding component of the 
the differential of the return of the renormalized diffeomorphism
is equal to zero.

Let us resume the proof. 
By taking appropriate renormalization, 
we can assume ($DT$2), ($DT$3), ($DT$4) hold for $f_2$. 
Let us consider the differential of the return map $F_2^N$.
Since $W^s(X)$ and $W^u(X)$ are tangent at $Q$, we can put
\[
dF_2^N(Q) =
\begin{pmatrix}
a      &  d      & g \\
b      &  e      & h \\
c      &   f      & 0
\end{pmatrix}.
\]
By applying Lemma \ref{Franks} at $P$, 
we can find a diffeomorphism $f_3$ arbitrarily close to $f_2$ such that 
the $c$, $f$, $g$ are not equal to zero and $(3, 3)$-component remains to be zero.
Note that the use of the Franks' lemma does not disturb the condition that $P, Q$
are contained in  $W^s(X) \cap W^u(X)$ (by letting the support of perturbation 
sufficiently small).
Let us take sufficiently large $l$ and consider the differential $dF_3^l \circ dF_3^{N}(Q)$.
This is equal to the following:
\[
\begin{pmatrix}
\lambda^{l}      & 0                                    & 0 \\
0                        & \tilde{\lambda}^{l}      & 0 \\
0                        & 0                                    & \mu^{l} 
\end{pmatrix}
\begin{pmatrix}
a      &  d      & g \\
b      &  e      & h \\
c      &   f      & 0
\end{pmatrix}
=
\begin{pmatrix}
\lambda^{l}a                  &  \lambda^{l}d                 & \lambda^{l}g             \\
\tilde{\lambda}^{l}b      &  \tilde{\lambda}^{l}e      & \tilde{\lambda}^{l}h \\
\mu^{l}c                        &  \mu^{l} f                        & 0
\end{pmatrix}.
\]
By applying Lemma \ref{Franks}, at $f_3^{l-1}(Q)$,
we perturb $f_3$ to $f_4$ to make the differential 
$dF_4(f^{l-1}(P))$ into the following form:
\[
\begin{pmatrix}
1                        & 0                                   & z_l \\
x_l                      & 1                                   &  y_l \\
0                        & 0                                   & 1
\end{pmatrix}
\begin{pmatrix}
\lambda             & 0                                   & 0 \\
0                        & \tilde{\lambda}            &  0 \\
0                        & 0                                   & \mu
\end{pmatrix},
\]
where 
\[ x_l := -(\tilde{\lambda}/\lambda)^l(h/g), \,\, y_l := -(\tilde{\lambda}/\mu)^l(ah/cg) -(\tilde{\lambda}/\mu)^l(b/c), \,\, z_l := -(\lambda/\mu)^l(a/c).\]
Then, by a direct calculation 
we have $(1, 1)$ $(2, 1)$, $(2, 3)$,  and $(3, 3)$-component 
of $dF^{N+l}_4(P)$ is equal to zero.
Note that the support of perturbation can be taken arbitrarily small, 
and the distance between $f_3$ and $f_4$ can be arbitrarily small 
if we take large $l$, since $x_l, y_l, z_l \to 0$ when $l \to \infty$.

Let us proceed our perturbation. By taking renormalization, we can put
\[
dF^N_4(P) := 
\begin{pmatrix}
0      &   b       & e \\
0      &   c       & 0 \\
a      &   d       & 0
\end{pmatrix},
\]
where $a$, $b$, etc. are not necessarily equal to the 
corresponding numbers appeared in $dF_3$.
Since $f_4$ is a diffeomorphism, we know $a, c, e \neq 0$.

Let us take a large integer $l$ and consider the differential $dF_4^N \circ dF_4^l (f_4^{-l}(P))$.
This matrix is given as follows:
\[
\begin{pmatrix}
0      &   b      & e \\
0      &   c      & 0 \\
a      &   d      & 0
\end{pmatrix}
\begin{pmatrix}
\lambda^{l}      & 0                                    & 0 \\
0                        & \tilde{\lambda}^{l}      & 0 \\
0                        & 0                                    & \mu^{l} 
\end{pmatrix}
=
\begin{pmatrix}
0                          &   b\tilde{\lambda}^{l}      & e \mu^{l} \\
0                          &   c\tilde{\lambda}^{l}      & 0               \\
a\lambda^{l}       &   d\tilde{\lambda}^{l}     & 0
\end{pmatrix}.
\]
Again by using Lemma \ref{Franks} at $f_4^{-l}(P)$, 
we perturb $f_4$ to $f_5$ to change the differential $dF_5(f_4^{-l}(P))$  into
\[
\begin{pmatrix}
\lambda             & 0                                   & 0 \\
0                        & \tilde{\lambda}            &  0 \\
0                        & 0                                   & \mu
\end{pmatrix}
\begin{pmatrix}
1                        & x_l                                   &   0\\
0                        & 1                                   &   0 \\
0                        & y_l                                  & 1
\end{pmatrix},
\]
where $x_l := -(\tilde{\lambda}/\lambda)^l(d/a)$ and $y_l := -(\tilde{\lambda}/\mu)^l(b/e)$.
Then a direct calculation shows the non-zero components of 
$df^{N+l}_5(f_5^{-l}(P))$ are only $(1, 3)$, $(2, 2)$, and $(3, 1)$.

Then, by taking renormalization, we can put
\[
dF^N_5(P) := 
\begin{pmatrix}
0      &   0      & c \\
0      &   b      & 0 \\
a      &   0      & 0
\end{pmatrix}.
\]

Now, using Lemma \ref{llin-map} at $F_5^N(P)$, 
we can construct $F_6$ such that $F_6^N$ is locally affine map
around $P$ such that its differential coincides with $dF_5^N(P)$ .

Let us review the properties of $f_6$.
 \begin{itemize}
\item There exists a coordinate chart $\phi : U \to \mathbb{R}^3$ around $X$ 
such that $\phi(U) = (-1,1)^3$ and $\phi(X)$ is the origin of $\mathbb{R}^3$.
\item  Let us put $F_6 :=  \phi \circ f_6 \circ \phi^{-1}$.
For  $(x, y, z) \in (-1,1) \times (-1 , 1) \times (-\mu^{-1}, \mu^{-1})$, 
$F_6 (x, y, z) =(\lambda x, \bar{\lambda} y , \mu z )$,
where $0< \tilde{\lambda} < \lambda < 1 <\mu$. 
\item There exist two points $P, Q \in U$ with $\phi(P) =(0, 0, p)$ and 
$\phi(Q) =(0, q, 0)$, where $0 < p, q < 1$ such that following holds:
For some positive integer $N \geq 2$, $f_6^N(Q) =P$ and $f_6^{i}(P) \not\in U$ for $0 < i  < N$. 
\item There exists a small neighborhood $W_0$ of $P$ with
$\phi(W_0) = [-\varepsilon,\varepsilon] \times [ -\varepsilon, \varepsilon]
\times [p- \varepsilon , p+ \varepsilon]$ and $F_6^N(W_0) \subset \phi(U)$
such that  for every $(x, y, z +p) \in W_0$, $F_t^N(x, y, z +p) = (az, by +q, cx)$, 
where $a, b$ and $c$ are non-zero real numbers.
\end{itemize}

Let $\rho_2(s)$ be a $C^{\infty}$-function on $\mathbb{R}$ 
satisfying the following properties:
$\rho_2(s) =\delta$ if $|s| < \varepsilon/3$ and
$\rho_2(s) =0$ if $|s| > 2\varepsilon/3$,
where $\delta$ is a positive real number.
Then define $R : V\to \mathbb{R}$ by $R(x, y, z) := \rho_2(x)\rho_2(y)\rho_2(z-p)$,
and construct a one-parameter family of the diffeomorphims $f_{7, t}$ as follows:
For $X \in W_0$,   $f_{7, t}(x, y, z) = F_6(x +tR(x, y, z), y, z)$.
Otherwise, $f_{7, t}(X) = f_6(X)$. 
Note that for $t$ sufficiently close to $0$, 
$f_{7, t}$ is a diffeomorphism of $M$.

Then, by taking $W_0$ sufficiently small if necessary, 
we can see that $(f_{7,t})_{|t| < \delta}$, $P$, $Q$, $\phi$, $W_0$
satisfy ($DT$1)--($DT$5) for small $\delta$ 
and $f_7$ can be taken arbitrarily close to $f$. 
Hence the proof is completed.
\end{proof}

\begin{remark}
The geometric idea behind the perturbatioins from $f_2$ to $f_5$
is simple. Let us see that.

Since $Q$ is the point of tangency, the differential of $df^{N}(P)$ sends 
the tangent vector $v \in T_P(\phi(V))$ in the 
direction of $z$-axis parallel to $xy$-plane.
By a small perturbation, we can assume $df^N(P)(v)$ is not parallel
to $y$-direction ($f_1$ to $f_2$) 

Let us consider the differential $dF_2^{N+l}(P)$. 
Take the two-dimensional plane passing $P$ and parallel to the 
$xz$-plane, and let us consider its image under $dF_2^{N+l}(P)$.
If $l$ is sufficiently large, under some generic assumption ($f_2$ to $f_3$), 
this image turns to be almost parallel to the $xz$-plane. 
Thus, by a small perturbation,
we can assume that the image is parallel to $xz$-plane ($f_3$ to $f_4$).

By a similar argument for $f^{-1}_4$, we see that $y$ direction 
turns to be almost invariant under long time transision.
Hence by giving small perturbation, we can find a diffeomorphisms 
that preserves the $y$-direction ($f_4$ to $f_5$).
\end{remark}

\subsection{Proof of Proposition \ref{laffbif}}
Finally, let us give the proof of Proposition \ref{laffbif}.

\begin{proof}[Proof of Proposition \ref{laffbif}]
Put $I_n := [-\varepsilon, \varepsilon] \times [q-\varepsilon, q+\varepsilon]
\times [(p-\varepsilon)/\mu^n ,(p+\varepsilon)/\mu^n]$ and $t_n = p/ \mu^n$.
For $n$ sufficiently large, we show that $f_{t_n}$ has a hyperbolic periodic point
$R_n \in I_n$ with period $n+N$ such that $f_{t_n}$ has a heterodimensional cycle
associated to $X$ and $R_n$.

Let us take a point $A \in I_n$ and put $A = (x,y,z)$.
Then $F_{t}^n(A) = (\lambda ^n x, \tilde{\lambda}^n y, \mu^n z)$.
Note that $F_{t}^n(A)$ belongs to $W_0$ if $n$ is sufficiently large.
In the following, we assume this condition holds.
By the definition of the return map, we can see 
$F^{n+N}(A) =(c(\mu^n z -p), b\tilde{\lambda}^n y +q, a\lambda ^n x +t )$.
Suppose that this is a periodic point of period $n+N$ for
$t =t_n$. Then the following equalities hold:
\[
c(\mu^n z -p) =x, \quad b\tilde{\lambda}^n y +q = y, \quad  a\lambda ^n x + p/\mu^n = z.
\]
By a direct calculation, we get
\[ x=0, \quad y=q/(1-b\tilde{\lambda}^n), \quad z=p/\mu^n. \]
We put $y_n :=q/(1-b\tilde{\lambda}^n)$ and $z_n := p/\mu^n$.
Note that if we take sufficienly large $n$, the point 
$R_n := (0, y_n, z_n)$ is in $I_n$, since $R_n \to Q$ when $n \to +\infty$. 
Thus $R_n$ is indeed a periodic point of period $n+N$ for $n$ sufficiently large. 

Let us check that $R_n$ is a hyperbolic periodic point of index $2$.
The derivative of $f_{t_n}^n$ at $R_n$ is given by the matrix below:
\[
\begin{pmatrix}
\lambda^n & 0                            & 0 \\
0                & \tilde{\lambda ^n}  & 0 \\
0                & 0                             & \mu^n 
\end{pmatrix}.
\]  
The derivative of $f_{t_n}^N$ at $f_{t_n}^n(R_n)$ is given as follows:
\[
\begin{pmatrix}
0                & 0                            & c \\
0                & b                            & 0 \\
a                & 0                            & 0 
\end{pmatrix}.
\]  
Hence the Jacobian at $R_n$ is equal to 

\[
\begin{pmatrix}
\lambda^n & 0                            & 0 \\
0                & \tilde{\lambda ^n}  & 0 \\
0                & 0                             & \mu^n 
\end{pmatrix}
\begin{pmatrix}
0                & 0                            & c \\
0                & b                            & 0 \\
a                & 0                            & 0 
\end{pmatrix}
=
\begin{pmatrix}
0                            & 0                                    & \lambda^nc \\
0                            & \tilde{\lambda ^n}b       & 0 \\
a \mu^n                & 0                                     & 0 
\end{pmatrix}.
\]  
Now, a direct calculation shows that the eigenvalues of this matrix are given by  $\tilde{\lambda ^n}b$
and $\pm \sqrt{ac\lambda^n \mu^n}$. The absolute value of
$\tilde{\lambda ^n}b$ is less than one when $n$ is sufficiently large,
and the absolute value of $\pm \sqrt{ac\lambda^n \mu^n}$ are 
greater than one when $n$ is sufficiently large,
since $|\tilde{\lambda}|<1$ and $|\mu\lambda|>1$
(the second inequality is the consequence of the fact that 
 $X$ is volume expanding).
 
 \begin{figure}[t]
\begin{center}
  \includegraphics[width=4in]{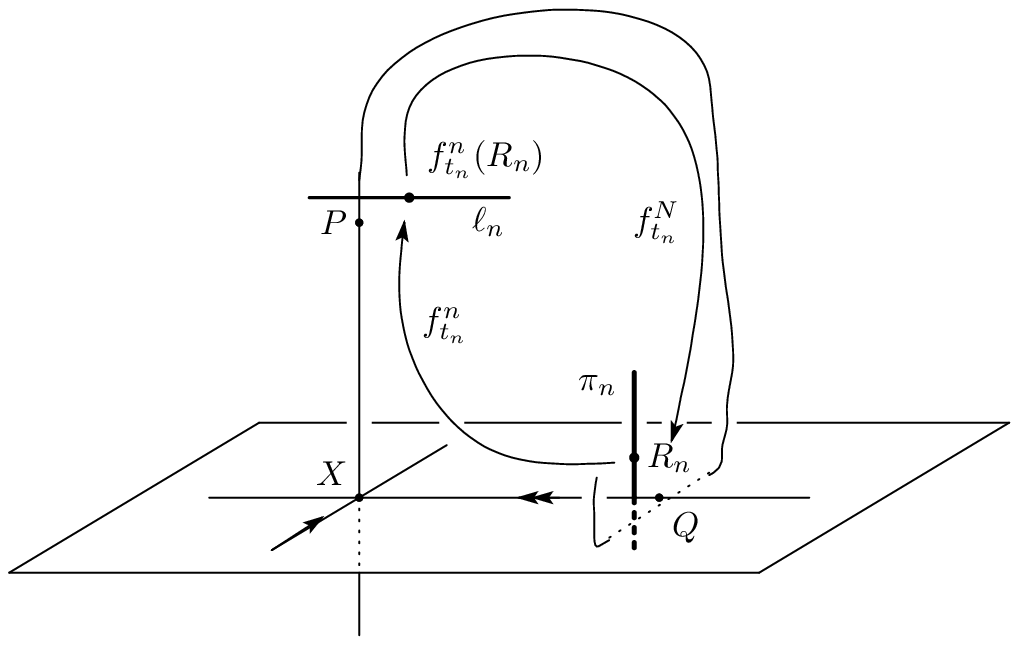}\\
 \caption{Creation of heterodimensional cycle.}
 \label{fig:DTbif}
  \end{center}
\end{figure}

Let us check that $f_{t_n}$ has a heterodimensional cycle
associated to $X$ and $R_n$ (see figure \ref{fig:DTbif}).
First, we show $W^u(X)$ and $W^s(R_n)$ have non-empty intersection.
It is easy to see that $W^u(X)$ contains $z$-axis in $\phi(U)$. 
Hence we only need to check that $W^s(R_n)$ has an intersection with $z$-axis.
To see this, we focus on the segment
\[
\ell_n := \{ (0 , \tilde{\lambda}^ny_n +s , \mu ^n z_n) \mid |s| \leq 2\tilde{\lambda}^n  y_n \}.
\]
The segment $\ell_n$ passes $f^n_{t_n}(R_n)$ and $z$-axis. 
We can see that $\ell_n$ is contained in $W_0$ if $n$ is sufficiently large,
since $f_{t_n}^{n}(R_n) \to P$ and the length of $\ell_n$ converges to zero
when $n \to +\infty$.
The image of this segment under $f_{t_n}^N$ is given as follows:
\[
f_{t_n}^N (\ell_n) = \{ ( 0, y_n +s, z_n ) \mid  |s| \leq 2|b| \tilde{\lambda}^n  y_n  \}.
\]
So the image of $\ell_n$ under $f_{t_n}^{n+k}$ is given as follows:
\[
f_{t_n}^{n+N}(\ell_n) = \{ (0 , \tilde{\lambda}^ny_n +s , \mu ^n z_n) \mid |s| \leq (2|b|\tilde{\lambda}^n)\tilde{\lambda}^n  y_n \}.
\]
Thus the restriction of $f_{t_n}^{n+N}$ to $\ell_n$  
is well defined and $\ell_n$ is uniformly contracted with the factor
$|b|\tilde{\lambda}^n$. 
This shows that for sufficiently large $n$, $\ell_n$ is contained in the stable manifold of $R_n$.

Second, let us check that the two invariant manifolds $W^s(X)$ and $W^u(R_n)$
have non-empty intersection.
As is in the previous case, we only need to 
check that $W^u(R_n)$ has an intersection with $xy$-plane in $U$.
To see this, we focus on the segment
\[
\pi_n := \{ (0,  y_n, z_n +s) \mid  |s| \leq 2 z_n \}.
\]
This is a segment that passes $R_n$ and have non-empty intersection with $W^s(X)$.
We show that  for sufficiently large $n$, $\pi_n$ is contained in $W^u(R_n)$.
To see that, let us calculate the inverse image of $\pi_n$ under $f_{t_n}^{-2(n+N)}$.

The inverse image of $\pi_n$ under $f_{t_n}^{-N}$  is given as follows:
\[
f_{t_n}^{-N}(\pi_n) = 
\{ (  s,  \tilde{\lambda}^ny_n,  \mu ^n z_n ) \mid  |s| \leq 2z_n/c\}.
\]
This set is contained in $W_0$ if $n$ is sufficiently large.
The inverse image of $f_{t_n}^{-N}(\pi_n)$ under $f^{-n}$  is given as follows:
\[
f_{t_n}^{-n-N}(\pi_n) = \{ (  s,  y_n, z_n ) \mid  |s| \leq 2z_n/(|c|\lambda^n)\}.
\]
Since $2z_n/(|c|\lambda^n) = 2p/(|c|\mu^n \lambda ^n)$ and $\mu\lambda >1$,
this segment is also contained $W_1$ when $n$ is sufficiently large.

The inverse image of $f_{t_n}^{-n-N}(\pi_n)$ under $f^{-N}$  is given as follows:
\[
f_{t_n}^{-n-2N}(\pi_n) = \{ ( 0,  \tilde{\lambda}^ny_n,  \mu ^n z_n +s ) \mid  |s| \leq 2z_n/(|ac|\lambda^n)\}.
\]
If $n$ is sufficiently large, this segment is also contained in $W_0$.

Finally, the inverse image of $f_{t_n}^{-n-2N}(\pi_n)$ under $f^{-n}$  is given as follows.
\[
f_{t_n}^{-2n-2N}(\pi_n) = \{ ( 0,  y_n,  z_n +s ) \mid  |s| \leq 2z_n/(|ac|\mu ^n\lambda^n)\}.
\]

These calculations show that $f_{t_n}^{-2n-2N}$ uniformly contracts $\pi_n$ by the factor $(|ac|\mu ^n\lambda^n)^{-1}$,
and its absolute value is less than $1$ if $n$ is sufficiently large.
Hence we know $\pi_n$ belongs to the unstable manifold of $R_n$. 
Thus the proof is completed.
\end{proof}

\section*{Acknowledgments} The author is very thankful for 
the good patience and the warm encourgement of 
his supervisor Shuhei Hayashi during preparing the manuscript.

\end{document}